\documentclass[12pt,a4paper,oneside]{amsart}

 \usepackage{amsfonts, amsmath, amssymb, amsthm, amscd, hyperref, anysize, graphicx, cleveref, framed, xcolor, wrapfig, mathtools, thmtools}

\theoremstyle{definition}
\newtheorem{theorem}{Theorem}[section]
\newtheorem{lemma}[theorem]{Lemma}

\newtheorem{conjecture}[theorem]{Conjecture}
\newtheorem{problem}[theorem]{Problem}

\newtheorem{proposition}[theorem]{Proposition}

\newtheorem{definition}[theorem]{Definition}
\newtheorem{question}[theorem]{Question}

\theoremstyle{remark}
\newtheorem{remark}[theorem]{Remark}
\newtheorem{example}[theorem]{Example}

\numberwithin{equation}{section}

\DeclareMathOperator{\vol}{vol}
\DeclareMathOperator{\area}{area}

\DeclareMathOperator{\len}{len}
\DeclareMathOperator{\act}{\mathcal{A}}

\DeclareMathOperator{\dist}{dist}

\DeclareMathOperator{\conv}{conv}
\DeclareMathOperator{\pos}{pos}

\title{Triangle covering problems and the Viterbo inequality in the plane}
\author{Alexey Balitskiy}
\address{University of Luxembourg, Department of Mathematics, 6 avenue de la Fonte, L-4364 Esch-sur-Alzette, Luxembourg}
\email{alexey.balitskiy@uni.lu}
\author{Ivan Mitrofanov}
\address{Universit\'e de Gen\`eve, Section de math\'ematiques, 7-9 rue G\'en\'eral-Dufour, 1204 Gen\`eve, Switzerland}
\email{ivan.mitrofanov.math@gmail.com}
\author{Alexander Polyanskii}
\address{Emory University, Department of Mathematics, Atlanta, GA 30322, US}
\email{apolian@emory.edu}

\begin{document}
\begin{abstract}
We review a certain problem on covering triangles in the plane. Equivalently, it can be viewed as a family of `isobilliard' inequalities in convex shapes, and as a special case of Viterbo's conjecture in symplectic geometry. We give an elementary overview of these topics and, using the optics of the covering problem, we establish several new special cases of Viterbo's conjecture, provide a simple explanation of the counterexample of Haim-Kislev and Ostrover, and state a few open questions. The main novel result is a proof of Viterbo's conjecture for lagrangian products $K \times Q$, where $Q \subset \mathbb{R}^2$ is any quadrilateral and $K \subset \mathbb{R}^2$ is any convex shape.
\end{abstract}
\maketitle

\section{A glance: triangle covering problems}\label{sec:intro}

An engineer has a few triangular rulers and wants to keep them neatly in a flat folio. The shapes of the rulers are known, as are the engineer's peculiar preferences: they want the folio to be convex, and they refuse to flip or rotate the rulers; the rulers may only be moved by parallel translation. What are the minimal area of such a folio and its optimal shape? This is the question we will address. Formally, given a collection of triangles in the plane, we are looking for a way to parallel translate them, with overlaps allowed, in order to minimize the area of their convex hull. Let us mention some concrete instances before we state the problem abstractly and explain where the triangles come from. 

\begin{enumerate}
    \item Engineer D.~Isk has infinitely many rulers: one for each triangle with perimeter $2$, including degenerate ones, i.e., ordinary straightedges of length $1$. This is equivalent to \emph{J.~Wetzel's worm problem} \cite{wetzel1973sectorial}, and it is still open. We mention the current progress in \Cref{sec:history}.
    \item \begin{minipage}[t]{0.6\linewidth}
        \vspace{-8pt}
        Engineer Q.~Uadrilateral has four rulers, depending on two parameters $a,b>1$, as shown in the figure on the right.
        We solve this problem in \Cref{sec:quadrilateral-covering}, establishing a new special case of Viterbo's conjecture (more about it below).
    \end{minipage}%
    \hfill
    \begin{minipage}[t]{0.38\linewidth}
        \vspace{-8pt}
        \centering
        \includegraphics[width=\linewidth]{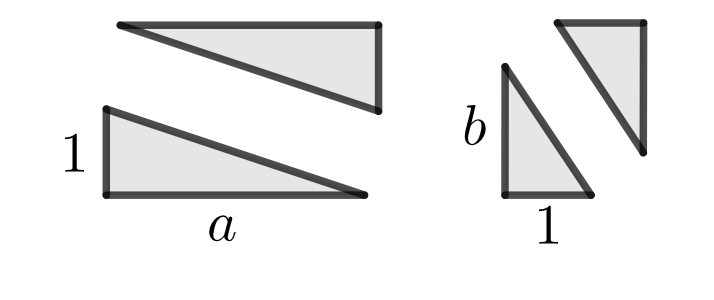}
    \end{minipage}
    \item Engineer P.~Entagon has ten rulers: one is a triangle with angles $\pi/5$, $\pi/5$, $3\pi/5$, and the others are its copies rotated by $k\pi/5$, $k = 1,\ldots,9$. An attempt to solve this problem (\Cref{sec:pentagonal-covering}) leads to the counterexample of P.~Haim-Kislev and Y.~Ostrover~\cite{haim2026counterexample}, which disproves Viterbo's conjecture in its greatest generality.
\end{enumerate}

All these questions are instances of a single problem asking to find the \emph{smallest $Q$-cover} (\Cref{prob:cover}), where $Q$ can be any convex shape (by a convex shape, we mean any compact convex set in the plane with non‑empty interior). In the examples above $Q$ is a disk, or a quadrilateral, or a regular pentagon. 

We will briefly review the history of covering problems in \Cref{sec:history}, as well as the relation between the smallest $Q$-cover and the geometry of billiards and Reeb orbits, leading to a conjecture of C.~Viterbo. This conjecture is a natural question of isoperimetric flavor in symplectic geometry, and a version of it happens to be equivalent to our $Q$-covering problem. Without going into details, we state this version now, and in \Cref{sec:history} we will explain the terminology, including the symplectic capacity $c(\cdot)$ that we use.

\begin{conjecture}[False in this generality!]\label{conj:viterbo}
For any convex bodies $K, Q \subset \mathbb{R}^n$, their lagrangian product satisfies the `isocapacitary' inequality:
\[
\vol (K \times Q) \ge \frac{c(K \times Q)^n}{n!}.
\]
\end{conjecture}

In the meantime, the reader may think of the following `isobilliard' reformulation of this conjecture, which will be explained in \Cref{sec:history}: for any convex bodies $K, Q \subset \mathbb{R}^n$, in the ``billiard table'' $K$ there exists a periodic billiard orbit of length at most $\sqrt[n]{n! \vol (K \times Q)}$, where both billiard reflections and lengths are defined using the asymmetric norm on $\mathbb{R}^n$ whose dual unit ball is a translate of $Q$.

\Cref{conj:viterbo} is known to be true in several special cases~\cite{hermann1998non,karasev2019viterbo,balitskiy2020equality,abbondandolo2025symplectic}, but in the generality stated above it was disproved recently by P.~Haim-Kislev and Y.~Ostrover~\cite{haim2026counterexample}. However, their counterexample does not eliminate one of the main motivations behind the conjecture: a weaker symmetric variant of Viterbo's conjecture would still imply~\cite{artstein2014from} the longstanding Mahler conjecture~\cite{mahler1939ubertragungsprinzip} from 1939. 
In the sequel we refer to the inequality from \Cref{conj:viterbo} as the \emph{Viterbo inequality}, meaning that this is a property that may or may not hold for particular products.

In this paper, we work in the case $n=2$. In \Cref{sec:reduction} we explain the equivalence of \Cref{conj:viterbo} and the problem of the smallest $Q$-cover (\Cref{prob:cover}). It boils down to an elementary question about decomposing (in some way) a closed curve into triangles, whose shapes align well with $Q$ (\Cref{def:interleave}).

In \Cref{sec:triangular-covering,sec:quadrilateral-covering,sec:pentagonal-covering,sec:hexagonal-covering} we consider the problem of the smallest $Q$-cover in the following cases:
\begin{itemize}
    \item $Q$ is a triangle---as a warm-up and a showcase of our techniques;
    \item $Q$ is a quadrilateral---as the main non-trivial result of this paper;
    \item $Q$ is a pentagon---as a simple explanation of why Viterbo's conjecture is false;
    \item $Q$ is a hexagon---as a source of additional tricks, proofs, and open questions.
\end{itemize} 

When $Q$ is any quadrilateral or a regular hexagon, we completely describe all smallest $Q$-covers. The corresponding result on the symplectic side is two-fold. First, we establish the Viterbo inequality for two new classes of examples.

\begin{theorem}\label{thm:quadri-hexa-viterbo}
\Cref{conj:viterbo} holds if $K \subset \mathbb{R}^2$ is any convex shape and $Q \subset \mathbb{R}^2$ is either an arbitrary quadrilateral or an affinely regular hexagon:
\[
\area K \cdot \area Q \ge \frac12 c(K \times Q)^2.
\]
\end{theorem}

Second, we get a characterization of all equality cases: all convex shapes $K$ such that $\area K \cdot \area Q = \frac12 c(K \times Q)^2$, when $Q$ is a quadrilateral or an affinely regular hexagon. It turns out to be quite tricky. For example, in the problem of the smallest folio of Q.~Uadrilateral, if $\frac1a + \frac1b \neq 1$, then there is a unique equality case, and if $\frac1a + \frac1b = 1$, then there is a one-parameter family of such equality cases. When $Q$ is an affinely regular hexagon, there are two two-parameter families of equality cases as well as two isolated ones.

The counterexample of \cite{haim2026counterexample} raises a natural question of how much the Viterbo inequality should be relaxed in order to be true (\cite[Question~2]{haim2026counterexample}). It is known~\cite{artstein2008m} that \Cref{conj:viterbo} holds up to a factor of an absolute constant raised to the power of $n$. We state this question in our special case of interest, to which our methods apply.

\begin{question}\label{ques:constant-viterbo}
What is the smallest number $\kappa$ such that for any convex shapes $K, Q \subset \mathbb{R}^2$, the relaxed Viterbo inequality 
\[
\area K \cdot \area Q \ge \frac1{2\kappa} c(K \times Q)^2
\]
holds?
\end{question}

It follows from \cite{haim2026counterexample} that $\kappa \ge \frac{\sqrt 5 + 3}{5} > 1$. As a corollary of the Viterbo inequality for quadrilaterals, we will see that $\kappa < \sqrt 2$.

\subsection*{Acknowledgments} This material is based in part upon work supported by the HORIZON EUROPE Research and Innovation programme under Marie Skłodowska-Curie grant agreement No. 101107896. Alexander Polyanskii is partially supported by NSF grant DMS 2349045. Ivan Mitrofanov is supported by SNSF grant TMAG-2\_216487\slash1.

\section{A chronicle: from worms through billiards to symplectic systoles}\label{sec:history}

\subsection{Worms}
While in this paper we focus on \emph{translation covers}, the history of covering problems starts with \emph{congruence covers}, also known as \emph{displacement covers} or \emph{universal covers}. More specifically, given a collection of shapes $\{T_j ~\vert~ j \in J\}$ to be covered, one is asked to find the smallest (in some sense) shape $K$ such that for any $j \in J$ there is an isometry $f_j$ of the euclidean plane such that $f_j(T_j) \subset K$. In other words, $K$ contains a congruent copy of each $T_j$. Apparently, the first covering problem is due to H.~Lebesgue, who asked in 1914 about the area-minimizing congruence cover for all sets of unit diameter (\cite[Chapter~11.14, Problem~1]{brass2005research}). The question is still open; see \cite{brass2005lower,gibbs2018upper} for the current best bounds. 

If all $f_j$ are required to be translations, a shape $K$ for which one can choose translations $f_j$ with $f_j(T_j) \subset K$ is called a \emph{translation cover} for the family $\{T_j ~\vert~ j \in J\}$. The first famous question on translation covers was asked by S.~Kakeya in 1917: what is the smallest area of a shape that contains a translated copy of every ``needle'' (any line segment of unit length)? The solution among convex shapes was shown by J.~Pál to be an equilateral triangle~\cite{pal1921minimumproblem}, whereas Besicovitch constructed non-convex covers of arbitrarily small area~\cite{besicovitch1928kakeya}.

A famous unsolved problem, known as \emph{L.~Moser's worm problem}~\cite[Chapter~11.4, Problem~4]{brass2005research}, is to find the area-minimizing congruence cover for the collection of curves of length $1$. In~\cite{wetzel1973sectorial}, J.~Wetzel proved a lower bound for the area of such a congruence cover, which is still unbeaten. More importantly for us, in the same paper Wetzel proposed a related problem: find the smallest area of a convex \emph{translation} cover for all \emph{closed} curves of length $1$. We call it \emph{Wetzel's worm problem}, though given that the curves are closed, we could have called it a ringworm problem as well. This problem is open as well~\cite[Chapter~11.4, Problem~8]{brass2005research}, with the best lower bound due to Wetzel himself~\cite{wetzel1973sectorial}. A classical observation due to V.~Klee~\cite{klee1953critical}, following directly from Helly's theorem, leads to an equivalent reformulation of Wetzel's worm problem~\cite[Remark~1]{bezdek1989covering}, which we stated in the introduction: a convex set $K$ is a translation cover for any closed curve of unit length if and only if it is a translation cover for any triangle of unit perimeter (including degenerate ones, which are line segments of length $1/2$), so one might look for the smallest translation cover for triangles of unit perimeter. The best (in terms of area) known example of such a set $K$ (and the best upper bound for the area in Wetzel's problem) is due to H.~Mallée~\cite{mallee1994translationsdeckel}, and we describe this example now. In \Cref{fig:mallee}, left, one sees a parabola with a vertical directrix (not shown) and focus $F$, so that every ray of light passing through $F$ reflects off the parabola and exits horizontally. An arc of this parabola can be reflected under the $D_3$ dihedral group to obtain the shape $K$, shown in \Cref{fig:mallee}, right, circumscribed around a regular triangle (one of the dashed triangles) having $F$ as one of its vertices and with its opposite side horizontal. Inside $K$ one sees three dashed closed curves---two triangles and one line segment traversed twice---of equal length.

\begin{figure}[ht]
  \centering
  \begin{minipage}{0.48\textwidth}
    \centering
    \includegraphics[width=\linewidth]{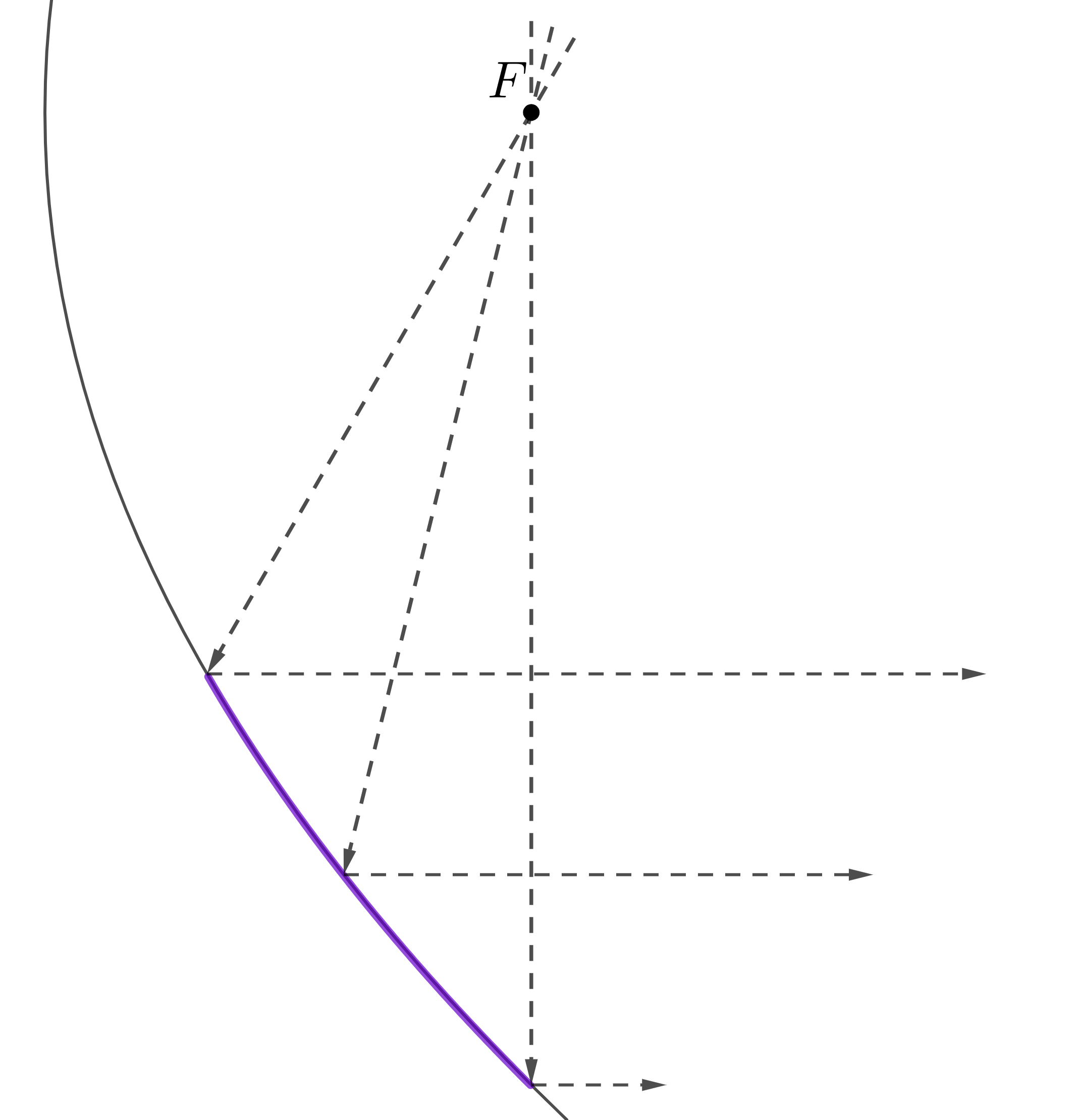}
  \end{minipage}
  \hfill
  \begin{minipage}{0.48\textwidth}
    \centering
    \includegraphics[width=\linewidth]{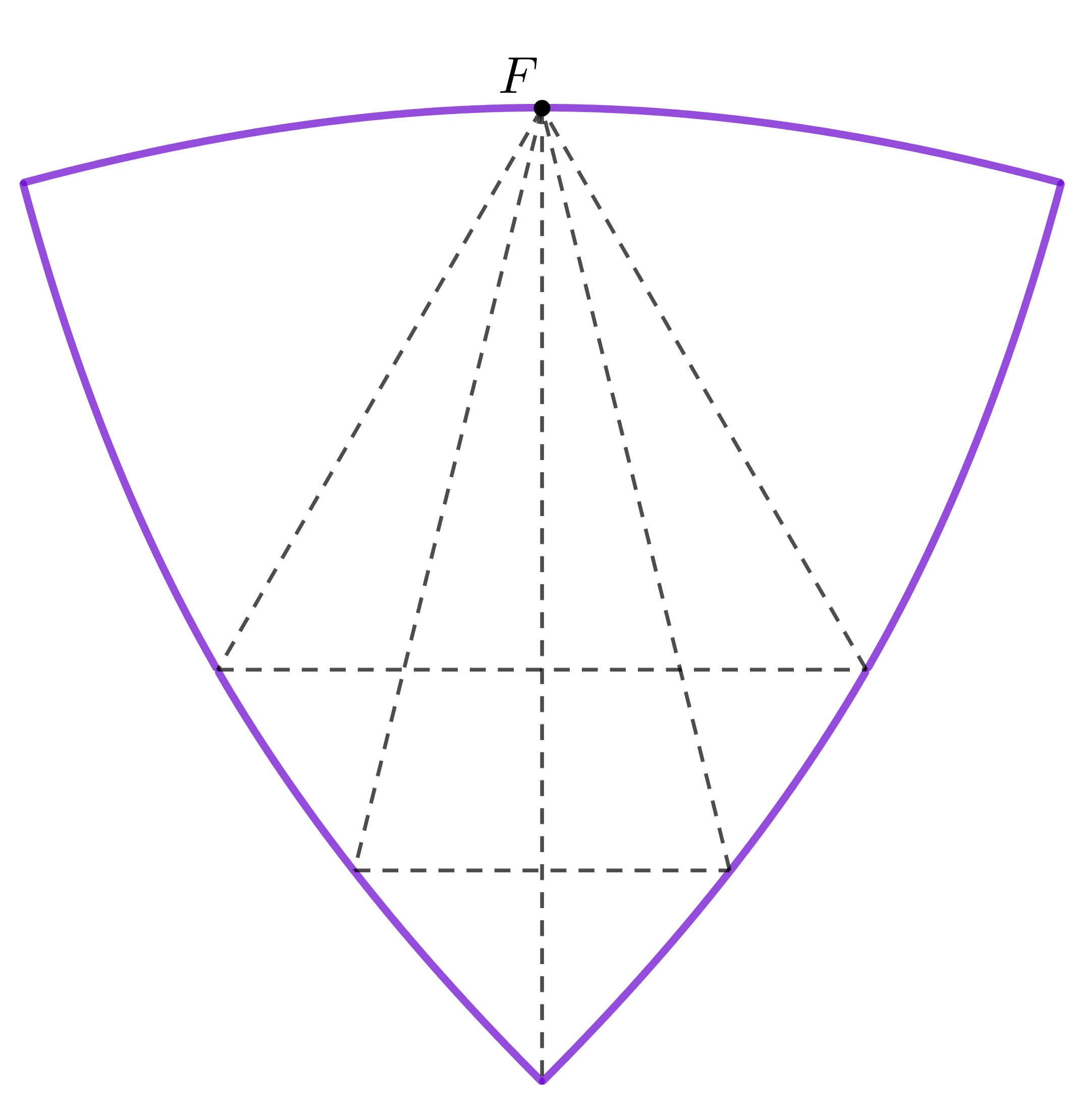}
  \end{minipage}
  \caption{Mallée's parabolic triangle (right) is composed of six parabolic arcs. One of these arcs is shown on the left, and the full shape is obtained by reflecting it under the $D_3$ dihedral group around the central regular triangle.}
  \label{fig:mallee}
\end{figure}

We note that the corresponding problem of finding the smallest \emph{congruence} cover for all triangles of fixed perimeter is solved in~\cite{furedi2000smallest}. We recommend \cite{wetzel2003fits} and \cite[Chapter~11.4]{brass2005research} for further information on planar covering problems. Among recent advances, we mention the version of Wetzel's problem that only considers centrally symmetric closed curves; it is solved in~\cite{jung2023universal}.

In \Cref{sec:reduction} we will pose a version of Wetzel's worm problem in which both the way lengths are measured and the collection of triangles to be covered depend on a parameter---a convex polygon $Q$. Some instances of this problem were given in the introduction.

\subsection{Billiards}

An important insight of \cite{bezdek1989covering} is an equivalent reformulation of Wetzel's worm problem in terms of billiards.\footnote{People had considered billiard orbits in the context of covering closed curves earlier, e.g. in~\cite{chakerian1973minimal}, but to our knowledge \cite{bezdek1989covering} was the first to give an equivalent reformulation of Wetzel’s worm problem in billiard terms.} The idea is to consider the triangle (perhaps, degenerate) of shortest perimeter that cannot be covered by a translate of the interior of $K$; it turns out it can be translated so that all its vertices lie on the boundary of $K$, and its sides are related to the boundary in a way that is akin to the billiard reflection rule. The only catch is that the corresponding boundary point does not have to belong to a smooth part of the boundary of $K$, so the billiard rule needs to be generalized. We say that a polygonal curve $u \to v \to w$, where $u,w \in K$, $v \in \partial K$, $u \neq v \neq w$, satisfies the \emph{generalized billiard reflection rule} at $v$, if there is a support line $\ell$ for $K$ at $v$ such that the outward normal vector $\nu$ to $\ell$ forms equal angles with the vectors $v-u$ and $w-v$. The central objects of this subsection are \emph{closed}, or \emph{periodic}, generalized billiard orbits. We say that a billiard orbit is \emph{$m$-periodic} if it has $m$ reflection points in $\partial K$. Now we can formulate the elegant result of K.~Bezdek and R.~Connelly~\cite[Lemma~8]{bezdek1989covering}, using modern terminology. 

\begin{theorem}\label{thm:bezdek}
    A convex shape $K$ is a translation cover for any closed curve of unit length if and only if $K$ does not admit a $2$- or $3$-periodic generalized billiard orbit of length less than $1$.
\end{theorem}


Let us denote by $c(K \times B)$ the length of the shortest closed generalized billiard orbit in the ``billiard table'' $K$. Here $B$ is the unit disk in the plane, and the rationale behind this notation will be explained a bit later. While the paper \cite{bezdek1989covering} only hints at considering closed generalized billiard orbits of arbitrary period, the later paper~\cite{bezdek2009shortest} shows---using similar ideas---that the shortest closed generalized billiard orbit always exists and is necessarily $2$- or $3$-periodic. Wetzel's worm problem now becomes an `isobilliard' problem: it asks to maximize the ratio $\frac{c(K \times B)^2}{\area K}$ over convex shapes $K$. Now we can state the conjectural optimality of Mallée's example as an isobilliard inequality.

\begin{conjecture}\label{conj:mallee}
Mallée's parabolic triangle has the smallest area among convex shapes $K \subset \mathbb{R}^2$ with $c(K \times B) = 1$.
\end{conjecture}

The next milestone in our story is the observation made in~\cite{akopyan2016elementary}: \Cref{thm:bezdek} remains valid when $\mathbb{R}^2$ is endowed with an arbitrary norm, with lengths and billiard reflections defined accordingly. Given a compact convex set $Q \subset \mathbb{R}^2$ containing the origin in its interior, one can define the (asymmetric) norm $\|\cdot\|_Q$ as the Minkowski functional associated to $Q^\circ$, the polar set of $Q$ (see \Cref{sec:reduction} for the details). This norm defines a non-euclidean way of measuring lengths, which we call $Q$-lengths. Thinking about light propagation in an anisotropic medium and employing Fermat's principle, one is led to an appropriately modified definition of billiard reflection, which gives rise to \emph{Minkowski billiards}, introduced by E.~Gutkin and S.~Tabachnikov~\cite{gutkin2002billiards}. The arguments of \cite{bezdek1989covering,bezdek2009shortest} can be adapted to this new setting, implying that in any convex billiard table $K$ the shortest closed (generalized) $Q$-Minkowski billiard orbit exists and is $2$- or $3$-periodic. We denote by $c(K\times Q)$ its length. The following result appears in~\cite{akopyan2016elementary},\footnote{The paper \cite{akopyan2016elementary} only works with (non-generalized) Minkowski billiards for smooth $K$ and $Q$, and \cite[Theorem~2.1]{akopyan2016elementary} is formulated there in the smooth case. A similar result in the case where $Q$ is smooth and origin-symmetric, so that the norm $\|\cdot\|_Q$ is symmetric and strictly convex, was independently obtained by Y.~Nir~\cite{nir2013closed}. Once an appropriate definition of generalized Minkowski billiards is adopted, the details of the non‑smooth case are straightforward and can be found in~\cite{rudolf2024minkowski}.} building on the ideas from~\cite{bezdek1989covering, bezdek2009shortest}.


\begin{theorem}\label{thm:normed-bezdek} 
A convex shape $K$ is a translation cover for any closed curve of unit $Q$-length if and only if $c(K\times Q) \ge 1$. 
\end{theorem}

We remark that even though we restrict our attention here to planar covers and billiards, the papers \cite{bezdek2009shortest,akopyan2016elementary} work in higher dimensions; the relevant higher-dimensional notions are straightforward generalizations of two-dimensional ones. 

The definition of $c(K\times Q)$ is invariant with respect to translations of $K$, and in \Cref{sec:reduction} we explain why one does not need to assume that the origin lies in the interior of $Q$, by defining $Q$-lengths in a way that is translation-invariant with respect to $Q$ as well. We also note that if $K$ or $Q$ is rescaled by a positive factor, then $c(K\times Q)$ gets multiplied by the same factor.

The notation $c(K\times Q)$, yet to be explained, might suggest that the roles of $K$ and $Q$ are somewhat symmetric. This is indeed the case; we outline the intuition very briefly. Let $q, q' \in \partial K$ be two consecutive points of reflection of a billiard orbit. In an anisotropic space, the velocity and the momentum are related as follows: the velocity of the segment $q \to q'$, positively proportional the vector $q'-q$, has to lie in $N_Q(p)$, the outward normal cone of $Q$ at the point $p \in \partial Q$, which is the momentum of $q \to q'$ (rescaling the velocity, we can assume that the momentum is unit). At the point $q'$, a billiard reflection occurs, and it turns out that the new momentum $p' \in \partial Q$ is related to the old one in a way that makes $p-p'$ lie in $N_K(q')$, the outward normal cone of $K$ at $q'$. This surprising symmetry between the reflection rule and the velocity-momentum relation has various applications, including the fact that $c(K\times Q) = c(Q \times K)$. See \Cref{fig:minkowski} for an example of a periodic billiard orbit. To summarize, the Minkowski billiard dynamics in the configuration $(K,Q)$ is equivalent---after reversing time---to the Minkowski billiard dynamics in the configuration $(Q,K)$, where $Q$ becomes the billiard table and $K$ determines the geometry. For the details, we refer to~\cite{artstein2014bounds,rudolf2024minkowski}.

\begin{figure}[ht]
    \centering
    \includegraphics[width=0.7\linewidth]{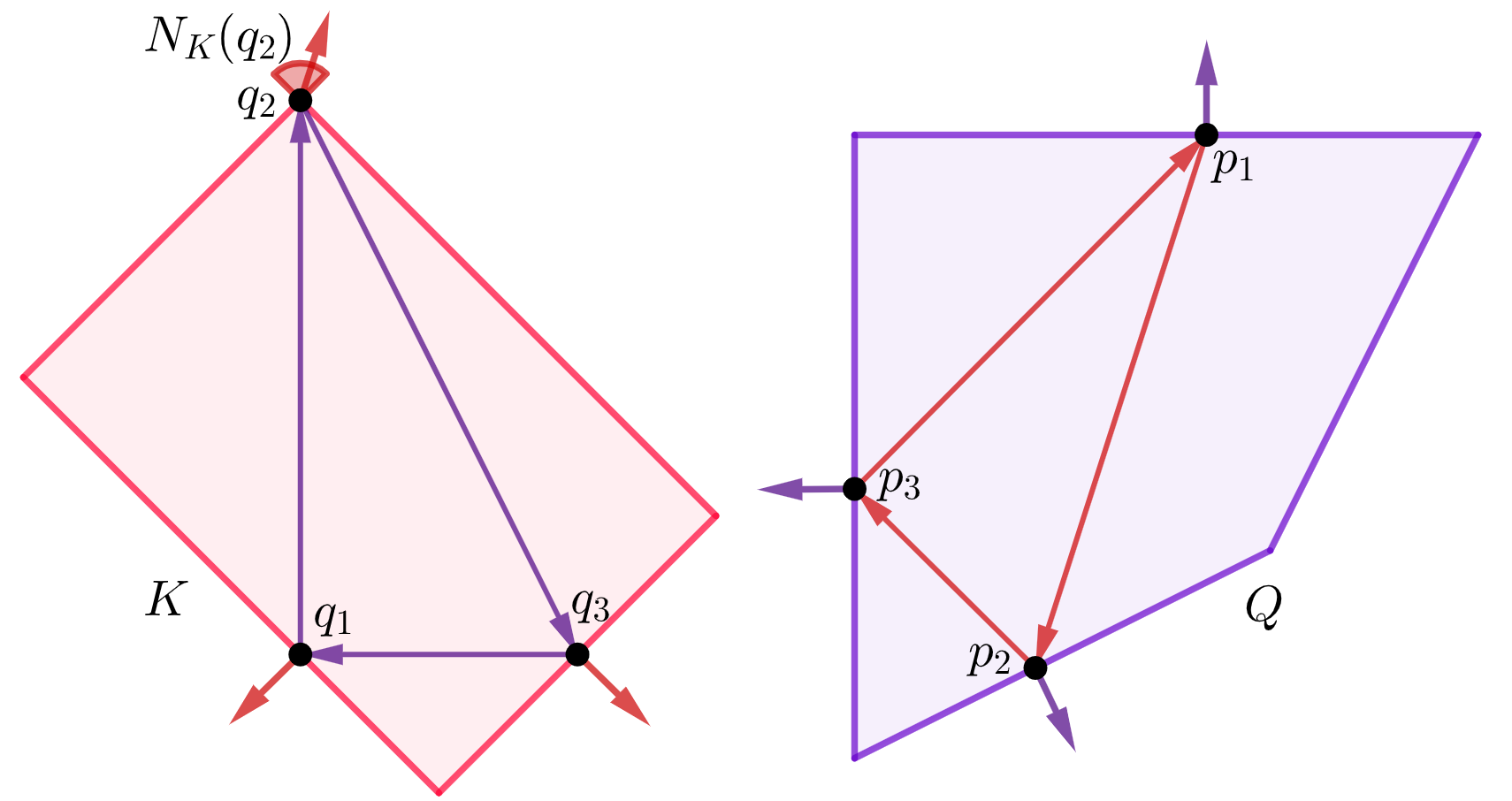}
    \caption{A periodic (generalized) $Q$-Minkowski billiard orbit in $K$ corresponds to a pair of closed polygonal curves in $K$ and $Q$ satisfying symmetric conditions: $q_{i+1} - q_i \in N_Q(p_i)$ and $p_{i-1} - p_{i} \in N_Q(q_i)$ for all $i$.}
    \label{fig:minkowski}
\end{figure}

\subsection{Symplectic systoles}

We explain now how Minkowski billiards fit in the framework of much more general symplectic dynamics. Again, for simplicity, we limit the discussion to the lowest-dimensional scenario, although all the notions can be considered in higher dimensions as well.

In this subsection, it is convenient to regard $\mathbb{R}^2 \times \mathbb{R}^2$ as a phase space with coordinates $(q_1,q_2,p_1,p_2)$. The billiard table $K$ lives in the $q$-space (space of positions), while the dual unit ball $Q$ of the norm lives in the $p$-space (space of momenta). The splitting of $\mathbb{R}^2 \times \mathbb{R}^2$ into the space of positions and the space of momenta is called \emph{lagrangian}, and so is the product $K \times Q$. We will see how the billiard dynamics in the configuration $(K,Q)$ corresponds to the symplectic dynamics on the boundary of their product, $\partial(K \times Q) \subset \mathbb{R}^2 \times \mathbb{R}^2$.

More generally, let $X \subset \mathbb{R}^2 \times \mathbb{R}^2$ be any convex body, the prototype example being $X = K \times Q$. Its boundary $\partial X$ is a convex hypersurface in $\mathbb{R}^2 \times \mathbb{R}^2$. If we pretend for a second that $\partial X$ is smooth, then we have two natural---and in fact equivalent---ways of defining smooth dynamics on it. 

\begin{enumerate}
    \item \underline{Contact viewpoint}. Consider the standard symplectic form $\omega = q_1 \wedge p_1 + q_2 \wedge p_2$. The restriction $\omega\vert_{\partial X}$ has a one-dimensional kernel, which can be oriented according to the sign of the Liouville form $\lambda = p_1 dq_1 + p_2 dq_2$, which endows $\partial X$ with a \emph{contact structure}. The integral curves of $\ker \omega\vert_{\partial X}$ are called \emph{characteristics}, or \emph{Reeb flow trajectories}, and the integral of $\lambda$ along a Reeb flow trajectory is called its \emph{action}.  
    \item \underline{Hamiltonian viewpoint}. View $\partial X$ as a non-degenerate level set of a smooth hamiltonian $H: \mathbb{R}^2 \times \mathbb{R}^2 \to \mathbb{R}$. The associated hamiltonian vector field $\mathfrak{X}$ is determined by $\omega(\mathfrak{X}, \cdot) = -dH(\cdot)$. This canonically defines characteristics as the trajectories of the corresponding hamiltonian flow. 
\end{enumerate}

The second approach is easier to extend to the convex non-smooth case, if we reformulate it as follows. It is convenient to introduce the ``imaginary unit'' operator $\mathcal{J}(q_1,q_2,p_1,p_2) = (-p_1,-p_2,q_1,q_2)$, so that the standard inner product on $\mathbb{R}^2 \times \mathbb{R}^2$ is given by $\langle x,y \rangle = \omega(x, \mathcal{J}y)$. Then characteristics can be equivalently defined as solutions of the classical Hamilton equation of motion $\dot{\gamma}(t) = -\mathcal{J} \nabla H(\gamma(t))$. The action can be rewritten as $\frac12 \int \langle \gamma(t), \mathcal{J}\dot{\gamma}(t)\rangle dt$. In the non-smooth case, the gradient direction $\nabla H$ is replaced by the outward normal cone of $K \times Q$, and the Hamilton differential equation becomes a differential inclusion. It turns out that the resulting definition of Reeb dynamics in $\partial(K \times Q)$ exactly matches the definition of Minkowski billiard dynamics~\cite{artstein2014bounds}, and the action of Reeb flow trajectories in $\partial(K \times Q)$ coincides with the $Q$-length of  corresponding billiard trajectory in $K$.

Having defined the Reeb flow, a natural question is whether \emph{periodic} Reeb flow trajectories exist. They are called \emph{Reeb orbits}. In the setting as above, for a smooth convex hypersurface, the existence was proven in~\cite{rabinowitz1978periodic,weinstein1978periodic}. The extension to the non-smooth case is explained in~\cite[Appendix]{artstein2014bounds}. Reeb orbits of smallest action are often called \emph{symplectic systoles}. Finally we are ready to define the mysterious quantity $c(\cdot)$, one of the many symplectic invariants known as \emph{symplectic capacities}.

\begin{definition}\label{def:capacity}
Given a convex body $X \subset \mathbb{R}^2 \times \mathbb{R}^2$, its \emph{Ekeland--Hofer--Zehnder symplectic capacity} $c(X)$ is the smallest action of a Reeb orbit in $\partial X$.
\end{definition}

Without going into details, we just mention that a vibrant field of research is dedicated to various symplectic capacities, starting from the celebrated Gromov non-squeezing theorem; see, e.g,~\cite{cieliebak2007quantitative}. \Cref{conj:viterbo} was stated in~\cite{viterbo2000metric} more generally, for any symplectic capacity of any convex body, and not just the EHZ capacity of lagrangian products. A major source of motivation for studying \Cref{conj:viterbo} comes from the fact that it implies~\cite{artstein2014from} Mahler's conjecture from 1939~\cite{mahler1939ubertragungsprinzip}, which bounds the volume product $\vol K \cdot \vol K^\circ$ of any convex origin‑symmetric body $K \subset \mathbb{R}^n$ from below by that of a cube. The proof of the implication was significantly simplified in~\cite{akopyan2016elementary}, where it was reduced to a worm problem: if $K$ is origin-symmetric, then every closed curve of $K^\circ$-length $4$ can be covered by a translate of $K$.\footnote{A very short argument for the euclidean version of this statement can be found in~\cite{chakerian1973minimal}; see also the references therein for the history of this statement, which goes back at least to B.~Segre. The proof in~\cite{akopyan2016elementary} of the normed-space version happens to be the same.} 

There are other active directions of symplectic research concerned with the equality cases in the Viterbo inequality. It happens that many of them exhibit various properties of \emph{symplectic balls} in different sense. One way a convex body $X$ may resemble a symplectic ball is when the interior of $X$ is symplectomorphic to the interior of a round ball in $\mathbb{R}^2 \times \mathbb{R}^2$; that is, when there is a diffeomorphism between them preserving the restriction of $\omega$. This is a strong property, and it was recently shown in~\cite{matijevic2025systolic} that it extends to a certain analogy between the boundary $\partial X$ and the standard round sphere $S^3$. The boundary $\partial X$ can be similar to the sphere $S^3$ in terms of its Reeb flow. The Reeb flow on the standard round sphere is exceptionally rigid: the entire $S^3$ decomposes into an $S^2$-worth of Reeb orbits of equal action (forming the classical Hopf fibration $S^3 \to S^2$). When $\partial X$ enjoys similar properties to some extent, one speaks of variants of the \emph{Zoll} property, which is straightforward to define when $\partial X$ is smooth, but becomes delicate in the non-smooth setting. If $\partial X$ is smooth, then the classical Zoll property forces the Viterbo inequality to become an equality~\cite{weinstein1974volume}. If $\partial X$ is non-smooth, as in our case of interest $X = K \times Q$, a strong variant of the Zoll property was proposed in~\cite{matijevic2025systolic}; it is not known whether it implies that the Viterbo inequality is attained with equality.
We will briefly touch on these phenomena when discussing specific equality cases. Some of them are known to be symplectic balls, and the proofs are generally difficult. Some of the others are expected to be, with supporting evidence coming from the fact that they satisfy a weak variant of the Zoll property~\cite[Theorem~1.5]{rudolf2022viterbo}. Still others are known not to be symplectic balls~\cite[Remark 1.8]{haim2026counterexample}. In \Cref{sec:triangular-covering,sec:quadrilateral-covering,sec:hexagonal-covering} we describe several families of lagrangian products $K \times Q$ that we believe to be symplectic balls. Some of us also believe that these families, up to swapping $K$ and $Q$, exhaust all symplectic balls among lagrangian products in dimension $2+2$.

To conclude our brief history survey, we summarize the connection between worms, billiards, and symplectic systoles, explained in this section.

\begin{theorem}\label{thm:capacity-bezdek} The following are equivalent for convex shapes $K$ and $Q$.
\begin{enumerate}
    \item $K$ is a translation cover for any closed curve of unit $Q$-length. 
    \item $K$ is a translation cover for any triangle of unit $Q$-perimeter. 
    \item $K$ does not admit a $2$- or $3$-periodic generalized $Q$-Minkowski billiard orbit with $Q$-length less than $1$.
    \item The shortest periodic generalized $Q$-Minkowski billiard orbit in the billiard table $K$ has $Q$-length at least $1$.
    \item The Ekeland--Hofer--Zehnder symplectic capacity of $K\times Q$ is at least $1$.
\end{enumerate} 
\end{theorem}

This theorem can be used as a tool for studying \Cref{conj:viterbo}: the Viterbo inequality holds for $\bullet\times Q$ if and only if every translation cover for the collection of triangles of unit $Q$-perimeter has area at least $\frac{1}{2 \area Q}$. In the following section we strengthen this criterion by further narrowing down the set of ``worms''.



\section{A pivot: reduction to normal triangles}\label{sec:reduction}

In this section we add to the list of equivalences from \Cref{thm:capacity-bezdek}, by modifying its second item in the case when $Q$ is a polygon. The set of ``worms'' will be narrowed down from all triangles of unit $Q$-perimeter to a finite collection of such (possibly degenerate) triangles whose sides are orthogonal to those of $Q$. First, we carefully define the \emph{$Q$-normed geometry}; then we introduce these distinguished triangles (\Cref{def:normal}); and finally we prove \Cref{thm:normal-cover}, which allows us to restate the Viterbo inequality as a worm problem for a finite collection of distinguished triangles (\Cref{prob:cover}). The material of this section can be generalized to higher dimensions in a straightforward way.


\subsection{Normed geometry and normal triangles} 

We work in the euclidean plane $\mathbb{R}^2$ with the standard inner product $\langle \cdot, \cdot \rangle$, and we endow it with an additional metric structure. 
Apart from the usual euclidean geometry, allowing one to measure usual euclidean lengths, we consider as well the \emph{$Q$-normed geometry}, meaning that we can also measure ``$Q$-lengths'' using the support function of $Q$:
\[
\|x\|_Q \coloneq \max_{y \in Q} \langle x , y \rangle.
\]
If $Q$ contains the origin in its interior, this is equivalent to saying that $\|\cdot\|_Q$ is the norm on $\mathbb{R}^2$ whose unit ball is the set $Q^\circ$, the \emph{polar} (or \emph{dual}) set of $Q$:
\[
Q^\circ \coloneq \left\{ x \in \mathbb{R}^2 ~\middle\vert~ \langle x,y\rangle \le 1 ~ \forall y \in Q \right\}.
\]
We do not require $Q$ to be origin-symmetric, so the norm is not necessarily symmetric: $\|x\|_Q$ may differ from $\|-x\|_Q$ (sometimes the term \emph{gauge} is used for asymmetric norms). 

However, even if we do not require $Q$ to contain the origin in its interior, the number $\|x\|_Q$ is well-defined although it might be negative. Nevertheless, the triangle inequality still holds for $\|\cdot\|_Q$:
\[
\|x_1 + x_2\|_Q = \max_{y \in Q} \langle x_1+x_2 , y \rangle \le \max_{y \in Q} \langle x_1 , y \rangle + \max_{y \in Q} \langle x_2 , y \rangle = \|x_1\|_Q + \|x_2\|_Q.
\]

The $Q$-length obviously changes if $Q$ is shifted. However, in the sequel we only care about $Q$-lengths of closed curves, and shifting $Q$ does not change those. For any $t \in \mathbb{R}^2$, and any closed polygonal curve $\gamma = (v_1 \to \ldots \to v_m \to v_{m+1} = v_1)$, we have: 
\begin{align*}
\len_{Q+t}(\gamma) &= \sum\limits_{i=1}^{m} \|v_{i+1}-v_i\|_{Q+t} \\
&= \sum\limits_{i=1}^{m} \max_{y \in Q+t} \langle v_{i+1} - v_i, y\rangle \\
&= \sum\limits_{i=1}^{m} \max_{y \in Q} \langle v_{i+1} - v_i, y+t\rangle \\
&= \sum\limits_{i=1}^{m} \max_{y \in Q} \langle v_{i+1} - v_i, y\rangle + \sum\limits_{i=1}^{m} \langle v_{i+1} - v_i, t\rangle \\
&= \sum\limits_{i=1}^{m} \|v_{i+1}-v_i\|_{Q} + \langle 0,t \rangle \\
&= \len_{Q}(\gamma).
\end{align*}

For rectifiable curves, the $Q$-length is defined by approximation with polygonal curves. Since we will be only interested in $Q$-lengths of closed curves, we can shift $Q$ as we want: for instance, we can assume that $Q$ contains the origin in its interior (so that $\|\cdot\|_Q$ becomes an asymmetric norm). Shifting $Q$ can also simplify calculations. 

\begin{example}\label{ex:shift-origin}
Let us compute the (oriented) $Q$-perimeter of a (non-degenerate) triangle, $\len_Q(u \to v \to w \to u)$. The only information about $Q$ that we need is the positions of the three support lines $\ell_{u\to v}$, $\ell_{v\to w}$, $\ell_{w \to u}$ of $Q$ corresponding to the directions $v-u$, $w-v$, and $w-u$, respectively. This means that $\ell_{u\to v} = \{y \in \mathbb{R}^2 ~\vert~ \langle v-u, y \rangle = \|v-u\|_Q\}$, and similarly for $\ell_{v\to w}$ and $\ell_{w \to u}$. Since the vectors $w-v$ and $u-w$ are not parallel for a non-degenerate initial triangle, the lines $\ell_{v\to w}$ and $\ell_{w \to u}$ intersect, and we may choose a shift vector $t$ so that they intersect at the origin. Then we have
\[
\len_{Q}(u \to v \to w \to u) = \|v-u\|_{Q+t} + 0 + 0 = \dist(\ell_{v\to w} \cap\ell_{w \to u}, \ell_{u\to v}) \cdot |v - u|,
\]
where $\dist(\cdot,\cdot)$ denotes the euclidean distance between a point and a line, and $|v-u|$ denotes the euclidean length of $[u,v]$. We can use the right-hand side as a formula for $Q$-perimeter expressed just in terms of euclidean distances. A similar argument for a line segment $[u,v]$ shows that, for a shift vector $t \in \ell_{v\to u}$,
\[
\len_{Q}(u \to v \to u) = \|v-u\|_{Q+t} + 0= \dist(\ell_{v\to u}, \ell_{u\to v}) \cdot |v - u|,
\]
where $\dist(\ell_{v\to u}, \ell_{u\to v})$ essentially measures the euclidean width of $Q$ in the direction $v-u$.
\end{example}

From now on, let $Q$ be a convex polygon. Each side of $Q$ determines an \emph{outward normal direction}, in the sense that all outward normal vectors to this side coincide up to rescaling (multiplication by a positive number), and we refer to equivalence classes of nonzero vectors modulo rescaling as directions. 


\begin{definition}\label{def:normal} Let $Q \subset \mathbb{R}^2$ be a convex polygon containing the origin in its interior.
A nondegenerate triangle will be called \emph{$Q$-normal} if one of the two ways of orienting its boundary gives a polygonal curve $u \to v \to w \to u$ satisfying the following properties:
\begin{itemize}
    \item the directions of $u \to v$, $v \to w$, and $w \to u$ are outward normal directions of $Q$;
    \item the $Q$-perimeter is $1$:
    \[
    \len_Q(u \to v \to w \to u) = \|v-u\|_Q + \|w-v\|_Q + \|u-w\|_Q = 1.
    \]
\end{itemize}
A line segment $[u,v]$, viewed as a degenerate triangle, will be called \emph{$Q$-normal} if
\begin{itemize}
    \item the directions of $u \to v$ and $v \to u$ are outward normal directions of $Q$;
    \item its back-and-forth $Q$-length is $1$:
    \[
    \len_Q(u \to v \to u) = \|v-u\|_Q + \|u-v\|_Q = 1.
    \]
\end{itemize}
When we talk about \emph{$Q$-normal triangles}, we mean both degenerate and non-degenerate ones.
\end{definition}

The word `normal' reflects both the fact that these curves follow outward normal directions of $Q$ and the normalization of their lengths in $Q$-normed geometry.
Obviously, the property of being $Q$-normal is preserved under translations, so we will oftentimes speak about translation classes of $Q$-normal triangles.

\subsection{Reduction to triangle covers}  
For two sets $A,B \subset \mathbb{R}^2$, let us say that $A$ \emph{fits into} $B$ if $A + t \subset B$ for some translation vector $t \in \mathbb{R}^2$. If $\gamma : I \to \mathbb{R}^2$ is a curve, and the image of $\gamma$ in $\mathbb{R}^2$ fits into $B$, then we will simply say that the curve $\gamma$ fits into $B$. The main result of this section is as follows.

\begin{theorem}\label{thm:normal-cover}
If every $Q$-normal triangle fits into $K$, then every closed curve of unit $Q$-length fits into $K$. Consequently, every $Q$-normal triangle fits into $K$ if and only if $c(K\times Q)\ge 1$.
\end{theorem}


In view of \Cref{thm:normal-cover}, \Cref{conj:viterbo}, when stated for a fixed polygon $Q$, reduces to the following covering problem associated with $Q$.

\begin{problem}\label{prob:cover}
Let $Q \subset \mathbb{R}^2$ be a convex polygon. Let us say that a convex shape $K \subset \mathbb{R}^2$ is a \emph{$Q$-cover} if every $Q$-normal triangle fits into $K$. What is the smallest $Q$-cover, in terms of area? In other words: given a collection of $Q$-normal triangles, one from each translation class, we would like to translate each of them in order to minimize the area of the convex hull. Is it true that this area is at least $\frac{1}{2 \area(Q)}$?
\end{problem}

\begin{remark}
More generally, when $Q$ is convex but not a polygon, a natural generalization of \Cref{def:normal} is obtained if one considers all outward normal directions at \emph{smooth} points of $\partial Q$. The corresponding version of \Cref{prob:cover} is still equivalent to \Cref{conj:viterbo}, but there collection of (translation classes of) $Q$-normal triangles is now infinite, so the reformulation does not seem to be a simplification. For example, when $Q$ is a euclidean disk, the corresponding problem of the smallest $Q$-cover is just Wetzel's original worm problem.
\end{remark}

Now we turn to the proof of \Cref{thm:normal-cover}. The key idea in the proof is cutting a curve at several places and rearranging the pieces to make a few simpler curves. The way a longer curve is built out of shorter ones will be called \emph{interleaving}. 

\begin{definition}\label{def:interleave}
Let $\theta = (\theta_1, \ldots, \theta_k)$ be a collection of continuous non-decreasing functions $\theta_i : [0,k] \to [0,1]$ such that
\begin{itemize}
    \item each $\theta_i$ interpolates between $\theta_i(0) = 0$ and $\theta_i(k) = 1$ in a piecewise differentiable way;
    \item $\theta_i' = 0$ or $\theta_i' = 1$ everywhere except finitely many points;
    \item $\sum\limits_{i=1}^k \theta_i(t) = t$ for all $t \in [0,k]$. 
\end{itemize}
In other words, $[0,k]$ splits into finitely many intervals, and on each of them all but one of the $\theta_i$s are constant, and the remaining one grows with derivative $1$. 
Let us define the \emph{$\theta$-interleaving} of curves $\gamma_1, \ldots, \gamma_k : [0,1] \to \mathbb{R}^2$ as the curve $\gamma: [0,k] \to \mathbb{R}^2$ given by
\[
\gamma(t) = \sum_{i=1}^k \gamma_i(\theta_i(t)), \quad t \in [0,k].
\]
\end{definition}

\begin{figure}[ht]
    \centering
    \includegraphics[width=0.7\linewidth]{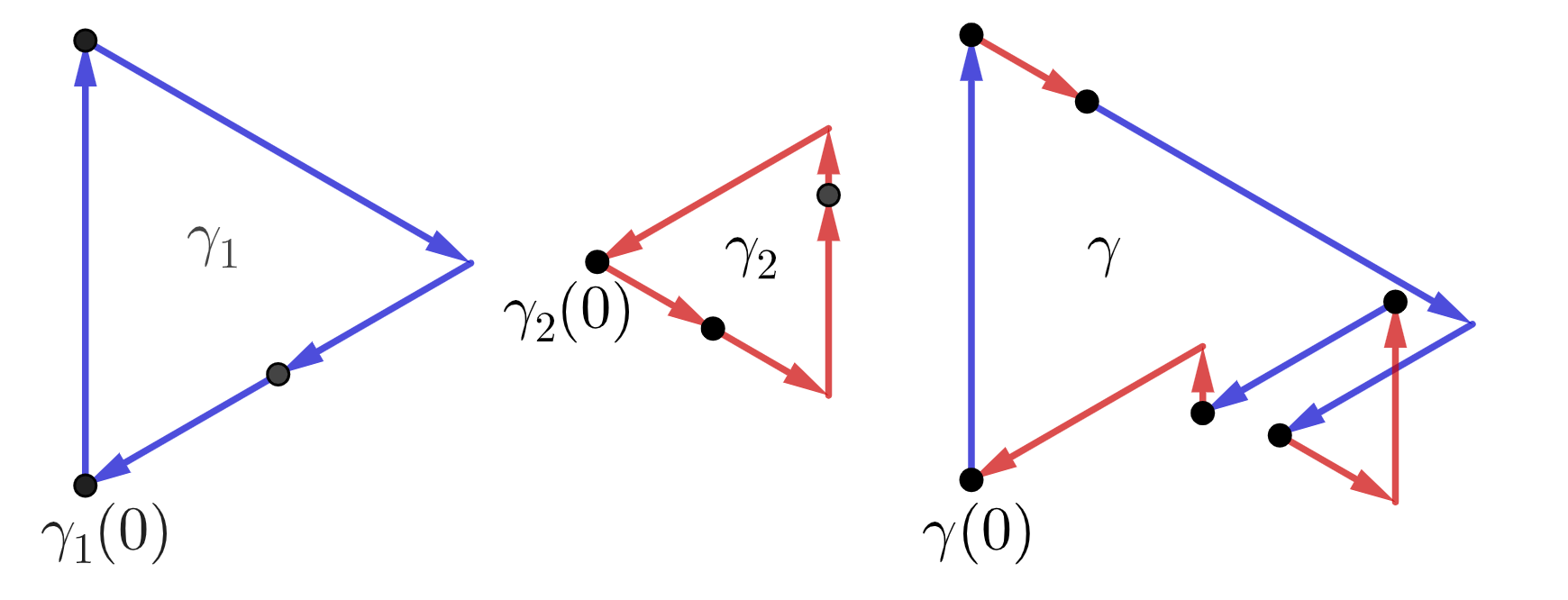}
    \caption{Two closed curves $\gamma_1$, $\gamma_2$, and their interleaving $\gamma$. Note that $\gamma_1$ and $\gamma_2$ are cut into three pieces each (at the black points), and then the pieces are interleaved to form $\gamma$. The order in which blue (resp. red) pieces appear in $\gamma$ is the same as in the original blue (resp. red) curve.}
    \label{fig:interleaving}
\end{figure}

Clearly, when all $\gamma_i$ are closed, their interleaving is closed too. The procedure of interleaving curves $\gamma_1, \ldots, \gamma_k$ can be described, up to reparametrization, as follows (see \Cref{fig:interleaving}). First, we cut $\gamma_1, \ldots, \gamma_k$ into several pieces. Then we form a new curve $\gamma$ by placing all these pieces one after another in some way, while preserving the relative order of the pieces coming from each individual curve $\gamma_i$. That is, for every $1 \le i \le k$, the pieces in $\gamma$ originating from $\gamma_i$ appear in $\gamma$ in the same order as they do along $\gamma_i$. From this description one can readily infer basic properties of interleaving such as `associativity up to reparametrization'. We state it for three curves (and implicitly use it later), but it can be stated analogously for any number of curves.

\begin{proposition}
Let $\gamma_1, \gamma_2, \gamma_3 : [0,1] \to \mathbb{R}^2$ be curves. Let $\gamma_{12} : [0,2] \to \mathbb{R}^2$ be an interleaving of $\gamma_1$ and $\gamma_2$, and denote by $\widetilde\gamma_{12} : [0,1] \to \mathbb{R}^2$ its reparametrization given by $\widetilde\gamma_{12}(t) = \gamma_{12}(2t)$. Then any interleaving of $\widetilde\gamma_{12}$ and $\gamma_3$ is, up to reparametrization, an interleaving of $\gamma_1, \gamma_2, \gamma_3$.
\end{proposition}



For the main proof of this section, it will be useful to relate Minkowski sums to convex hulls as follows.

\begin{lemma}\label{lem:minkowski-slice}
Let $A_i \subset \mathbb{R}^2$, $1\le i \le k$, be arbitrary sets, and let $\lambda_i \ge 0$ be coefficients of a convex combination, $\sum\limits_{i=1}^k \lambda_i = 1$. Then the Minkowski sum $\sum\limits_{i=1}^k \lambda_i A_i$ fits into the convex hull $\conv\left(\bigcup\limits_{i=1}^k A_i \right)$.
\end{lemma}
\begin{proof}
    The convex hull $\conv\left(\bigcup\limits_{i=1}^k A_i \right)$ can be viewed as the geometric join 
    \[
    A_1 * \ldots * A_k = \left\{\sum_{i=1}^k \mu_i x_i ~\middle\vert~ x_i \in \conv A_i, \mu_i \ge 0, \sum_{i=1}^k \mu_i = 1\right\}.
    \]
    As a subset of this, if we fix $\mu_i = \lambda_i$, we have the set $\sum\limits_{i=1}^k \lambda_i \conv A_i$, containing $\sum\limits_{i=1}^k \lambda_i A_i$. 
\end{proof}

\begin{proof}[Proof of \Cref{thm:normal-cover}]

Suppose we are given $Q$-normal triangles $T_1, \ldots, T_J \subset \mathbb{R}^2$, such that all translation classes of $Q$-normal triangles are present among them. Suppose also that we have a closed curve $\gamma$ of $Q$-length $1$. We need to argue that $\gamma$ fits into $K = \conv\left(\bigcup\limits_{j=1}^J T_j \right)$. 

\underline{Step 1}. First, we reduce to the case of polygonal curves in a straightforward way. By the definition of the $Q$-length of a general rectifiable curve $\gamma$, there exists a sequence of closed polygonal curves $\gamma^{(i)}$ such that the lengths $\len_Q(\gamma^{(i)})$ increase and tend to $\len_Q(\gamma) = 1$; and moreover, the convex hulls $\conv (\gamma^{(i)})$ converge in the Hausdorff metric to the convex hull $\conv (\gamma)$. Assuming the result has been proved for polygonal curves, we obtain vectors $t_i \in \mathbb{R}^2$ such that $\conv (\gamma^{(i)}) + t^{(i)} \subset K$. Since the sequence $(t^{(i)})$ is bounded, after passing to a subsequence we may assume that it converges to some vector $t \in \mathbb{R}^2$. Since $K$ is closed, we can pass to the Hausdorff limit in the inclusion $\conv (\gamma^{(i)}) + t^{(i)} \subset K$ and conclude that $\gamma$ fits into $K$. This reduces the case of general curves to the case of polygonal curves.

\underline{Step 2}. Now we reduce to the case of a curve only following outward normal directions of $Q$. We will replace $\gamma$ with such a curve, preserving the $Q$-length and increasing the convex hull.

Consider a line segment $u \to v$ of $\gamma$. Either it already follows an outward normal direction of $Q$, and there is nothing to be changed, or the vector $v-u$ belongs to the interior of the outward normal cone of $Q$ at some vertex. Let $\pos\{\nu, \xi \}$ be this cone, with $\nu$ and $\xi$ being outward normals of $Q$. Write $v-u = \alpha \nu + \beta \xi$ for some $\alpha, \beta > 0$. Replace the segment $u \to v$ in $\gamma$ with a two-segment curve $u \to u + \alpha \nu = v - \beta \xi  \to v$. This does not change the $Q$-length of the curve, since the functional $\|\cdot\|_Q$ is linear on $\pos\{\nu,\xi\}$.

Repeating this procedure over all segments of $\gamma$, complete this step. The resulting curve will be still called $\gamma$.

\underline{Step 3}. Now we will represent $\gamma$ as an interleaving of several curves $\gamma_1, \ldots, \gamma_M$, such that 
\begin{itemize}
    \item $\len_Q (\gamma_i) = \lambda_i$, $\sum\limits_{i=1}^M \lambda_i = 1 = \len_Q (\gamma)$;
    \item each $\gamma_i$ is triangle (perhaps, degenerate) tracing the boundary of a $\lambda_i$-rescaled copy of $T_{j(i)}$, for some $1 \le j(i) \le J$.
\end{itemize}

Let $\gamma = (v_1 \to \ldots \to v_m \to v_{m+1} = v_1)$; let $\nu_i = v_{i+1} - v_i$ be the corresponding outward normal vectors of $Q$, so that $\sum\limits_{i=1}^{m} \nu_i = 0$. Choose any inclusion-minimal subcollection of $\nu_i$ that admits a linear dependency with positive coefficients; by Carathéodory's theorem, it consists of two or three vectors. Let $\alpha_1 \nu_{i_1} + \ldots + \alpha_\ell \nu_{i_\ell} = 0$, $\ell=2$ or $3$, be this positive dependency. Assume that $\alpha_{1}$ is the largest coefficient. Now we can represent $\gamma$ as the interleaving of the following two curves.
\begin{enumerate}
    \item The curve $\gamma_1$ with $\ell$ segments is defined either as
    \[
    0 \to \nu_{i_1} \to \nu_{i_1} + \frac{\alpha_2}{\alpha_1} \nu_{i_2}= 0, \quad \text{if } \ell=2,
    \]
    or as
    \[
    0 \to \nu_{i_1} \to \nu_{i_1} + \frac{\alpha_2}{\alpha_1} \nu_{i_2} \to  \nu_{i_1} + \frac{\alpha_2}{\alpha_1} \nu_{i_2} + \frac{\alpha_3}{\alpha_1} \nu_{i_3} = 0, \quad \text{if } \ell=3.
    \]
    Note that $\gamma_1$ is a triangle (degenerate if $\ell=2$), and let $j(1)$ be its ``type'', that is, the index such that the $Q$-normal triangle $T_{j(1)}$ has the same shape up to rescaling by the factor of $\lambda_1 \coloneq \len_Q(\gamma_1)$.

    \item The curve $\gamma'$ can be read off from the following linear dependency with non-negative coefficients, obtained by subtracting $\frac{1}{\alpha_1}(\alpha_1 \nu_{i_1} + \ldots + \alpha_\ell \nu_{i_\ell}) = 0$ from $\sum\limits_{i=1}^{m} \nu_i = 0$:
    \[
    \sum_{i=1}^{m} \beta_i \nu_i = 0, \quad \text{where } \beta_i = 
    \begin{cases}
        1, & \text{if } i \notin \{i_1,\ldots, i_{\ell}\}; \\
        1 - \frac{\alpha_k}{\alpha_{1}}, & \text{if } i = i_k \text{ for some $k = 1,\ldots,\ell$}.
    \end{cases}
    \]
    The curve $\gamma'$ can now be defined as
    \[
    v_1 \to v_1 + \beta_1 \nu_1 \to v_1 + \beta_1 \nu_1 + \beta_2 \nu_2 \to \ldots \to v_1 + \sum_{i=1}^{m-1} \beta_i \nu_i \to v_1 + \sum_{i=1}^{m} \beta_i \nu_i = v_1.
    \]
    Note that at least one part of this polygonal curve collapses ($\beta_{i_{1}}=0$), so $\gamma'$ has fewer segments than $\gamma$.
\end{enumerate}
It is straightforward to see that the line segments of $\gamma_1$ and $\gamma'$ can be interleaved to obtain $\gamma$; in particular, $1 = \len_Q(\gamma) = \lambda_1 + \len_Q(\gamma')$. Hence, $\gamma$ is an interleaving of $\gamma_1$ and $\gamma'$. If $\lambda_1 = 1$, then $\gamma'$ is a trivial (stationary) curve, and we are done. Otherwise we can apply the same procedure to $\gamma'$, representing it as an interleaving of a triangle $\gamma_2$, similar to $T_{j(2)}$, and a curve $\gamma''$ that has fewer segments than $\gamma'$. Eventually we will decompose $\gamma$ into interleaving triangular curves $\gamma_1, \ldots, \gamma_M$.

Finally, since $\gamma_i$ fits into $\lambda_i T_{j(i)}$, it follows trivially from the definition of interleaving that $\gamma$ fits into $\sum\limits_{i=1}^M \lambda_i T_{j(i)}$.
Then by \Cref{lem:minkowski-slice}, $\gamma$ fits into $\conv\left(\bigcup\limits_{i=1}^M T_{j(i)} \right) \subset K$. This finishes the proof.
\end{proof}

\section{A warm-up: triangular norm}\label{sec:triangular-covering}

In this section we consider \Cref{prob:cover} in the case where $Q=\triangle$ is a triangle. This problem is trivial, and the answer can be found by a very elementary case analysis, but we use it as a playground to introduce some tricks that will help us later.

The first observation we make is that both volume and symplectic capacity are invariant under translations of $K$ or $Q$, as well as linear transformations
\[
K \times Q \mapsto AK \times (A^T)^{-1} Q.
\]
Any triangle is affinely equivalent to a regular one with unit area; therefore, we can assume that $\triangle$ is a regular triangle of unit width (unit euclidean distance between a vertex and its opposite side).

The $\triangle$-normed triangles come in two translation classes that differ by a rotation by $\pi$. Fix two model $\triangle$-normed triangles $T_y^0$ and $T_o^0$; we refer to them as `yellow' and `orange', in agreement with \Cref{fig:tri1}. A simple computation (\Cref{ex:shift-origin}) shows that they have unit (euclidean) side length. Denote by $T_y(t_y) = t_y + T_y^0$ the sliding yellow triangle, depending on a shift parameter $t_y \in \mathbb{R}^2$. Similarly, let $T_o(t_o) = t_o + T_o^0$ be the sliding orange triangle. The problem of the smallest $Q$-cover is equivalent to minimizing the area of $K \coloneq\conv (T_y(t_y) \cup T_o(t_o))$.

\begin{figure}[ht]
    \centering
    \includegraphics[width=\linewidth]{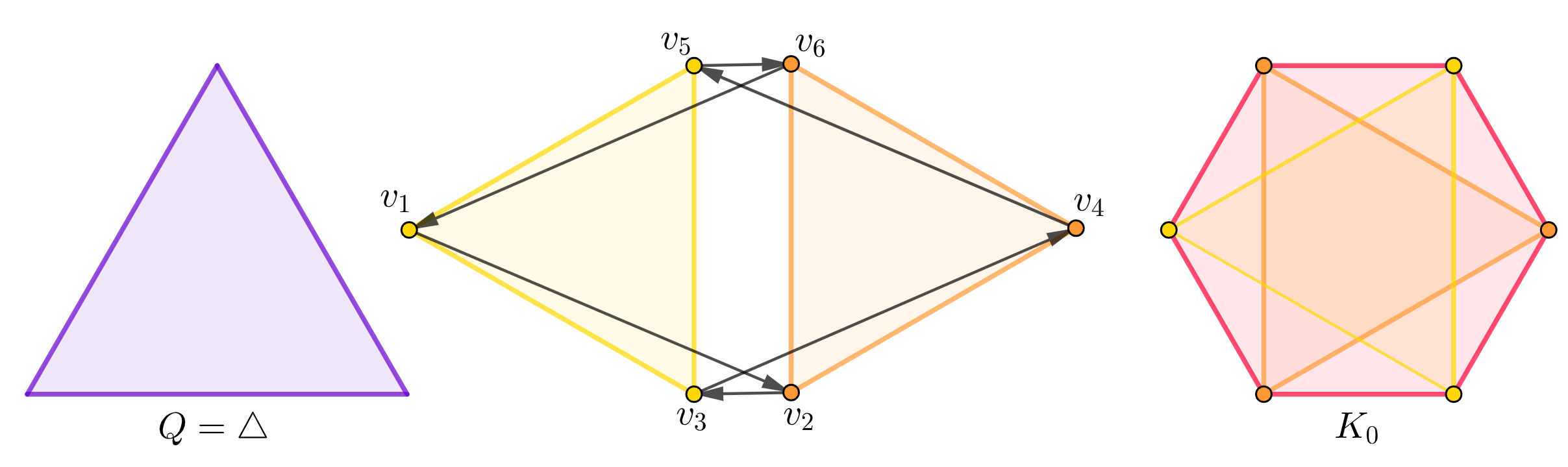}
    \caption{$\triangle$-normed triangles define the curve $\Gamma$ (middle), and in the most symmetric placement their convex hull is the regular hexagon $K_0$ (right).}
    \label{fig:tri1}
\end{figure}

We will prove with several different methods a sharp inequality $\area K \ge \area K_0$, where $K_0$ is the regular hexagon for the most symmetric placement of normal triangles. We will also see that $K_0$ belongs to a two-dimensional family of equality cases. The first method is the one that matters in the sequel, as it is particularly well suited to the study of equality cases in the Viterbo inequality in the plane. All three methods generalize to higher dimensions in one way or another. 

\subsection{First method: enclosing constant area}

Given two vectors $s = (s_1,s_2), y = (t_1,t_2)$ in the plane, we write $s \wedge t$ to denote the oriented area of the parallelogram spanned by $s$ and $t$:
\[
s \wedge t = \det \begin{pmatrix}
    s_1 & t_1 \\
    s_2 & t_2 
\end{pmatrix}.
\]

\begin{definition}\label{def:action}
Given a closed polygonal line $\gamma = (v_1 \to \ldots \to v_m \to v_{m+1}=v_1)$ in the plane, we introduce its \emph{enclosed area} by the formula
\[
\act(\gamma) = \frac12 \sum_{i=1}^m v_i \wedge v_{i+1}.
\]
\end{definition}

The geometric meaning behind this definition is as follows. The curve $\gamma$ splits the plane into regions $D_i$. Let $w_i \in \mathbb{Z}$ be the winding number of the curve $\gamma$ around $D_i$ (it can be negative). Then $\act(\gamma) = \sum\limits_i w_i \area D_i$. Another equivalent way of defining the enclosed area is via the line integral $\frac12 \int\limits_\gamma (x \,dy - y\, dx)$, which is connected to area by Green's formula. Note that $\act(\gamma)$ is invariant under translations of $\gamma$, and that it changes its sign if we reverse the orientation of $\gamma$.

There are two steps in this method. First, we find a closed polygonal curve $\Gamma \subset K$ connecting some of the points $v_i$ is a certain order, such that $\act(\Gamma)$ equals a positive constant $A$ that does not depend on the shift parameters $t_y$ and $t_o$. Second, we bound $\act(\Gamma) \le \area K$. If this inequality becomes an equality for certain values of the shift parameters, then the bound $\area K \ge A$ is sharp.

An appropriate curve $\Gamma$ for the first step is depicted in \Cref{fig:tri1}; the vertices of $T_y(t_y)$ are labeled $v_1(t_y),v_3(t_y),v_5(t_y)$, and the vertices of  $T_o(t_o)$ are labeled $v_2(t_o), v_4(t_o), v_6(t_o)$, as in the figure. Let us check that $\act(\Gamma(t_y,t_o))$ is a constant function. First, we compute its variation with respect to $t_y$:

\begin{align*}
\act(\Gamma(t_y+s,t_o)) - \act(\Gamma(t_y,t_o))
&= \bigl(v_1(t_y+s)\wedge v_2(t_o) - v_1(t_y)\wedge v_2(t_o)\bigr)
   + \text{five similar terms} \\
&= \sum_{k\in\{2,4,6\}}
   \bigl(s\wedge v_{k}(t_o) + v_{k}(t_o)\wedge s\bigr) \\
&= 0.
\end{align*}

Similarly, the variation with respect to $t_o$ is zero as well. Since this computation is valid at any $(t_y,t_o)$, $\act(\Gamma(t_y,t_o))$ is constant. The value $A$ of the constant can be computed if we choose the values of the shift parameters so that $K = K_0$ is a regular hexagon. A simple direct calculation shows that $A = \frac{1}{2 \area Q}$, so that the inequality we proved is exactly \Cref{conj:viterbo} in the case where $Q$ is a triangle.

At the second step, we want to make sure that $\act(\Gamma) \le \area \conv (T_y(t_y) \cup T_o(t_o))$. It suffices to check that the winding number of $\Gamma$ around any point never exceeds $1$. It should be easy for the reader to convince themselves this is true simply by dragging the triangles around and observing the shape of $\Gamma$. We will give an overkill argument (\Cref{lem:croissant}) that proves it for a large class of curves. This lemma will come especially handy in \Cref{sec:quadrilateral-covering}. In our current simple example, it concludes the proof if we set 
$a_i = v_i$, $b_i = v_{8-i}$, for $i = 1,2,3,4$ (with $v_7 = v_1$ by convention). The assumption of the lemma is satisfied because $a_1 - b_1 = 0 = a_4 - b_4$, and $a_2 - b_2 = v_2 - v_6 = v_3 - v_5 = a_3 - b_3$.

\begin{lemma}\label{lem:croissant}
    Let $\Gamma = (a_1 \to \ldots \to a_m \to b_m \to \ldots b_1 \to a_1)$ be a closed polygonal curve in the plane with the following properties:
    \begin{itemize}
        \item the vectors $a_i - b_i$ have the same direction for all $i = 1,\ldots,m$ (it is acceptable if $a_i=b_i$ for some $i$);
        \item the euclidean lengths $|a_i - b_i|$ form a unimodal sequence of numbers, that is, there is $j$ such that this sequence is non-decreasing for $1\le i \le j$ and non-increasing for $j \le i \le m$.
    \end{itemize}
    Then $\act(\Gamma) \le \area(\conv(\Gamma))$.
\end{lemma}

\begin{proof}
Note that this lemma is not obvious by default, and it fails for general curves. For example, if $\Gamma$ winds around a convex polygon $N$ times then $\act(\Gamma) = N \cdot \area(\conv(\Gamma))$.

Consider the Steiner symmetrization with respect to a line $L$ orthogonal to all $a_i - b_i$. For each $i$, it slides the points $a_i$ and $b_i$ along the line they belong to until the moment when the midpoint $m_i = \frac{a_i+b_i}{2}$ ends up on $L$. This procedure has two area-related properties.
\begin{itemize}
    \item The symmetrization does not change $\act(\Gamma)$, as it is easy to check directly from the definition. For $1 < i < m$, sliding $[a_i, b_i]$ by a vector $s \parallel b_i-a_i$ changes $\act(\Gamma)$ by the amount of $a_{i-1} \wedge s + s \wedge a_{i+1} + b_{i+1} \wedge s + s \wedge b_{i-1} = s \wedge (b_{i-1} - a_{i-1}) - s \wedge (b_{i+1} - a_{i+1}) = 0$. For $i=1$ or $i=m$ the computation is even simpler.
    \item The symmetrization can only decrease $\area(\conv(\Gamma))$. This can be seen by cutting $\conv(\Gamma)$ into trapezoids along the lines spanned by the pairs $(a_i,b_i)$, and observing how they transform. Here are the details of this argument. Let $\widetilde{a}_i$ and $\widetilde{b}_i$ denote the images of $a_i$ and $b_i$ after the symmetrization, and let $\widetilde{\Gamma}$ be the symmetrized curve. Let $\ell$ be a line perpendicular to $L$. The convex hull $\conv(\widetilde{\Gamma})$ is the union of all trapezoids $\widetilde{a}_i \widetilde{a}_j \widetilde{b}_j \widetilde{b}_i$. Each of these trapezoids either does not meet $\ell$, or intersects $\ell$ in a segment whose midpoint lies on $L$. Among these segments, the longest one contains all the others. Consider the corresponding trapezoid; before symmetrization it cut out on $\ell$ a segment of the same length, but the intersection $\ell \cap \conv(\Gamma)$ could have been longer. Thus the symmetrization can only decrease the length of the intersection $\ell \cap \conv(\Gamma)$. Applying the same argument to every line $\ell$, we conclude that the symmetrization can only decrease $\area(\conv(\Gamma))$.
\end{itemize}

Therefore, it suffices prove the lemma for $L$-symmetric curves. In \Cref{fig:croissant} one can see an example, where only one half of such a curve is depicted. The curve splits the half-plane into regions, which are labeled in \Cref{fig:croissant} with the multiplicity they contribute in $\act(\Gamma)$. These labels are equal to winding numbers of $\Gamma$ around regions. 

\begin{figure}[ht]
    \centering
    \includegraphics[width=0.7\textwidth]{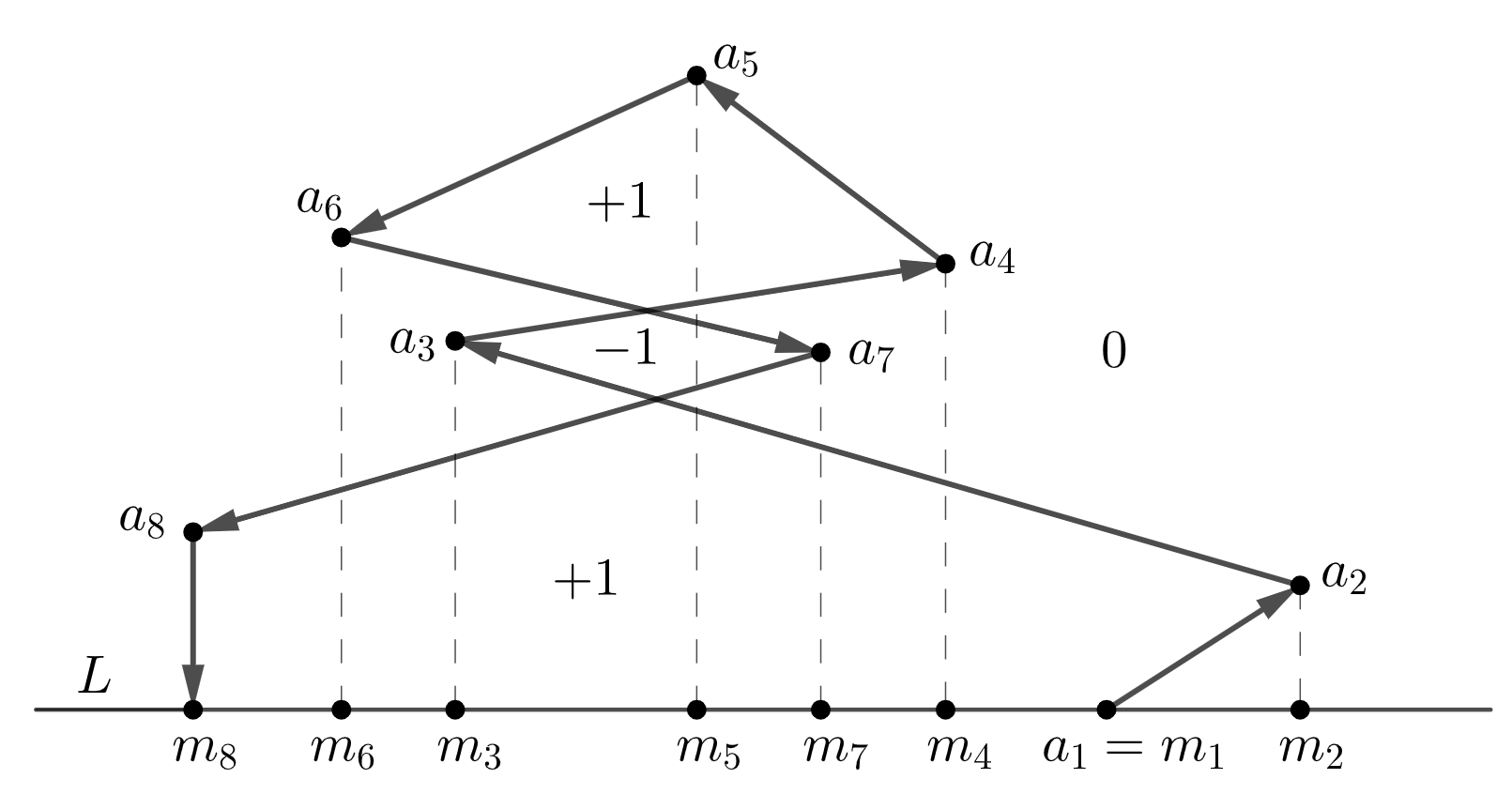}
    \caption{Half of the symmetrized curve $\Gamma$}
    \label{fig:croissant}
\end{figure}

Let us prove that no winding number can be greater than $+1$; this is enough to conclude the proof. Suppose the contrary: there is a point $p$ in the halfplane around which $\Gamma$ winds more than once. This would mean that there exist indices $i<j<k$ such that the segments $[a_i,a_{i+1}]$ and $[a_k,a_{k+1}]$ pass above $p$ from right to left, while $[a_j,a_{j+1}]$ passes below $p$ from left to right (here the words `left', `right', `above', `below' are understood as if $L$ is horizontal, as in \Cref{fig:croissant}). But this contradicts the unimodality assumption.
\end{proof}

\subsection{Second method: covering the plane by lattice translates}

This method generalizes well to higher dimensions: it was used in \cite{balitskiy2020equality} to prove \Cref{conj:viterbo} in the case when $Q$ is a simplex in $\mathbb{R}^n$ (and $K \subset \mathbb{R}^n$ is any convex body). We sketch it very briefly, since it will not be used in the rest of the paper.

As before, let $K = \conv (T_y(t_y) \cup T_o(t_o))$. Consider the lattice $\Lambda$ spanned by the sides of $T_y$ and $T_o$. Note that the plane admits a triangulation such that its vertex set is $\Lambda$, and every triangular cell of it is a $\triangle$-normal triangle. The key idea is to prove that the translates $K + \lambda$, $\lambda \in \Lambda$, cover $\mathbb{R}^2$. To show that, suppose there exists $x \notin K + \Lambda$. Then $(K - x) \cap \Lambda = \varnothing$, which contradicts the following lemma proven in \cite{balitskiy2020equality} by a simple topological argument.

\begin{lemma}[Planar case of {\cite[Lemma~5.3]{balitskiy2020equality}}]\label{lem:caging}
Suppose we are given a triangulation of the plane such that every triangular cell of it fits into a convex shape $K$. Then any translate of $K$ meets the vertex set of the triangulation.
\end{lemma}

Since $K + \Lambda = \mathbb{R}^2$, the volume of $K$ is at least that of any fundamental domain of $\Lambda$, for instance, the regular hexagon $K_0$, which is also the Voronoĭ cell of $\Lambda$. This lower bound on the area of $K$, when calculated explicitly, becomes our desired inequality $\area K \ge \frac{1}{2 \area \triangle}$.

\subsection{Third method: combining enclosing with lattice covering}

The proof following the second method can be completed in a different way. Instead of using \Cref{lem:caging}, one can substitute an idea from the first method.

Using the notation from the second proof, we define $\Lambda$ as the lattice spanned by the edges of $T_y$ or $T_o$. Our aim is to prove that the translates $K+\lambda$, $\lambda\in \Lambda$, cover $\mathbb R^2$.

For a closed polygonal curve $\gamma=(v_1\to \dots \to v_m\to v_{m+1}=v_1)$ in the plane, let $w_{\gamma}:\mathbb R^2\setminus\gamma\to \mathbb Z$ be defined as follows: $w_{\gamma}(x)$ is the winding number of the curve $\gamma$ around $x\in \mathbb R^2\setminus\gamma$.
Then $\mathcal A(\gamma)=\int_{\mathbb R^2} w_{\gamma}(x)\,dx$.

Let $\Gamma(t_y, t_o)$ be the closed polygonal line from the first proof. The key fact is that
\begin{equation}\label{eq:multiplicity}\tag{$\star$}
\sum_{\lambda\in \Lambda} w_{\Gamma(t_y+\lambda,t_o+\lambda)}(x) = 1    
\end{equation}
for any $x\not \in \bigcup\limits_{\lambda\in \Lambda} \Gamma(t_y+\lambda, t_o+\lambda)$. Note that the sum is well-defined since the support of $w_{\Gamma(t_y,t_o)}$ is bounded. The key fact can be verified as follows.
Every arrow in $\bigcup\limits_{\lambda\in \Lambda} \Gamma(t_y+\lambda, t_o+\lambda)$ appears together with its opposite arrow; it follows from the observation that the six arrows of $\Gamma(t_y,t_o)$ split into three pairs of opposite arrows, and in each pair $(v_i \to v_{i+1}, v_{i+3} \to v_{i+4})$ the difference $v_{i+3} - v_{i+1} = v_{i+4} - v_i$ belongs to $\Lambda$ (here we used cyclic indexing modulo $6$).
Therefore, the left-hand side of~\eqref{eq:multiplicity} equals the same integer $z$ wherever it is defined. This integer can be determined by integrating the left-hand side of~\eqref{eq:multiplicity} on a large square $S$ and comparing it to the area of $S$. Asymptotically, as the side length of $S$ tends to $\infty$, this integral equals $\act(\Gamma) \cdot |\Lambda \cap S|$, up to an error term depending on the perimeter of $S$ (again, since the support of $w_{\Gamma(t_y,t_o)}$ is bounded). Since $\act(\Gamma) = \area K_0$, and the $\Lambda$-translates of $K_0$ tile the plane, it follows that $z = 1$. In fact, to conclude the proof we do not need to know the exact value of $z$; it suffices to know that $z \neq 0$.

It is easy to see that if $w_{\Gamma(t_y,t_o)}(x)\ne 0$ for $x\not \in \Gamma(t_y,t_o)$, then $x\in K = \conv(\Gamma(t_y,t_o))$. By~\eqref{eq:multiplicity}, for each $x\not \in \bigcup\limits_{\lambda\in \Lambda}\Gamma(t_y+\lambda, t_o+\lambda)$, there is $\lambda\in \Lambda$ with $x\in K+\lambda$. By compactness of $K$, we conclude that the same is true for any $x \in \mathbb R^2$.

This proof can also be generalized to higher dimensions to give an alternative proof of \cite[Theorem~5.1]{balitskiy2020equality}.

\subsection{Characterization of equality cases}

The methods above provide different perspectives on smallest $\triangle$-covers. The analysis of equality cases in the first proof leads to the following conclusion: $\area K = \frac{1}{2 \area \triangle}$ if and only if $\Gamma$ is a convex curve (a curve winding once around a convex domain). The two-dimensional family of such equality cases, including $K_0$, is depicted in \Cref{fig:tri2}.
We call them \emph{tight hexagons}, associated with the three outward normal directions of $\triangle$ (most of them are indeed hexagons, but the boundary cases degenerate into parallelograms).

\begin{figure}[ht]
    \centering
    \includegraphics[width=0.75\linewidth]{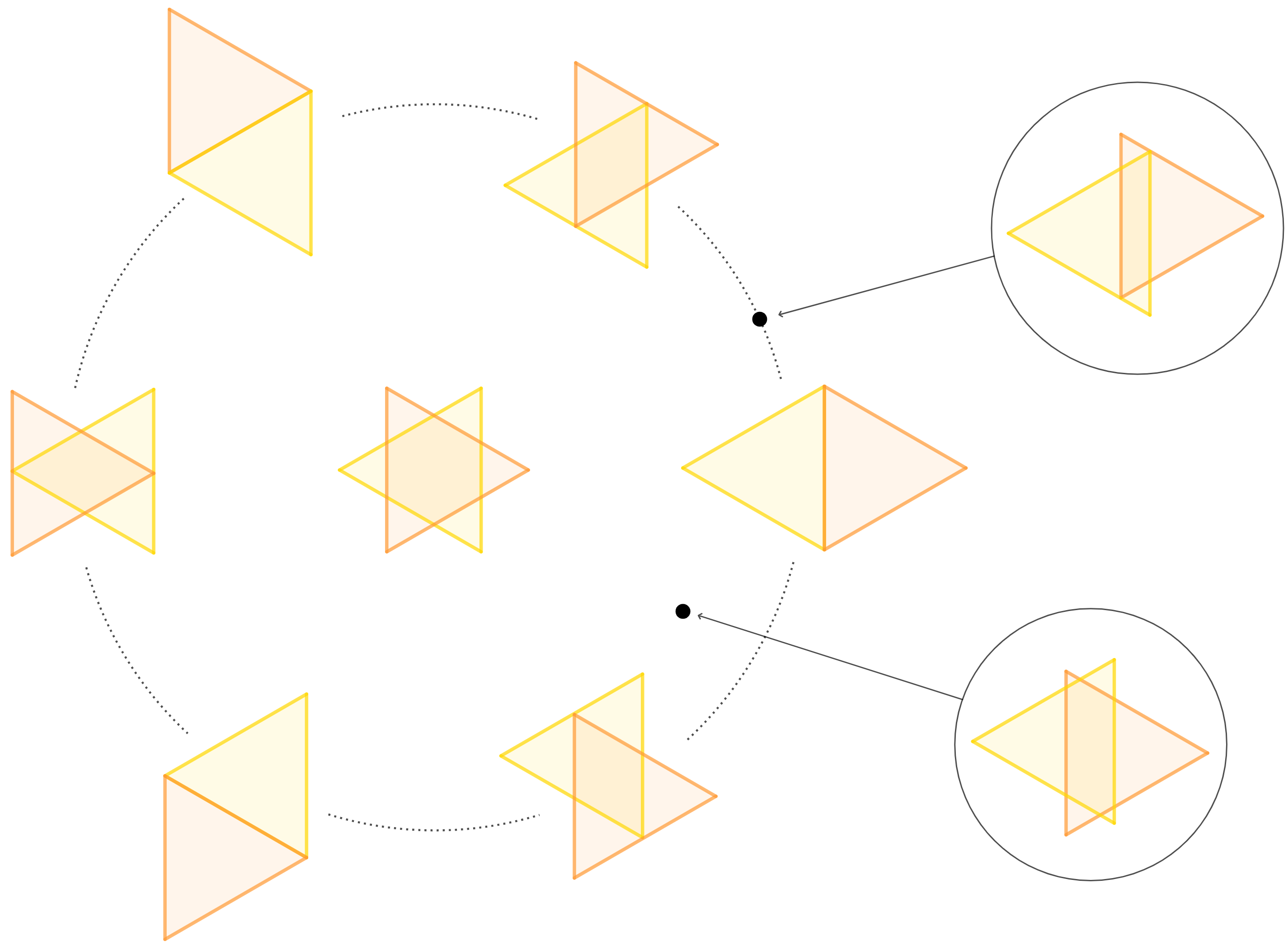}
    \caption{The two-dimensional family of area-minimizing $\triangle$-covers topologically is a disk. Each point represents a relative placement of the two $\triangle$-normal triangles, whose convex hull is a tight hexagon.}
    \label{fig:tri2}
\end{figure}

The second method gives the following perspective on equality cases: observe that $\area K = \frac{1}{2 \area \triangle}$ if and only if $K$ is a $\triangle$-cover that tiles the plane by $\Lambda$-translations. One way to describe them all is as follows. Consider the class $\Phi$ of positive definite quadratic forms $\phi(x)$ such that the Delaunay tessellation of $\Lambda$ with respect to the distance functional $\dist_\phi(x,y) = \sqrt{\phi(x-y)}$ coincides with the triangulation we considered above. The class $\Phi$ was described explicitly by Erdahl~\cite{erdahl1999zonotopes} as the forms 
$\phi(x) = \sum\limits_{i=1}^3 \alpha_i \langle d_i, x\rangle^2$,
where $d_1,d_2,d_3$ are vectors parallel to the sides of $Q$, and $\alpha_1,\alpha_2,\alpha_3$ are any positive numbers. Then any smallest $\triangle$-cover can be approximated by (the translates of) the Voronoĭ cells of $\Lambda$ with respect to $\dist_\phi$. This is the two-dimensional case of the characterization that works in any dimension, and it will be explained in a separate paper. 

Only one equality case from the family of tight hexagons has been shown to give rise to a symplectic ball. Namely, for the regular tight hexagon $K_0$, it was shown in~\cite{ostrover2025lagrangian}, using the Toda lattice, that the interior of $K_0 \times \triangle$ is symplectomorphic to a ball. It is an open question whether the same holds for other tight hexagons.

\subsection{Half-symmetric Viterbo inequality}

In this paragraph, $Q$ is any convex polygon (not a triangle). As we will see in \Cref{sec:pentagonal-covering}, \Cref{conj:viterbo} does not hold for general $K$ and $Q$. A simple argument shows that in the case where both $K$ and $Q$ are origin-symmetric, \Cref{conj:viterbo} is equivalent to Mahler's conjecture.\footnote{Here is this argument. The implication from Viterbo's conjecture to Mahler's conjecture is the main result of~\cite{artstein2014from}. For the reverse implication, inscribe in $K$ the largest possible homothetic copy $\beta Q^\circ$. It is known~\cite{artstein2014from} that, by connecting any pair of antipodal points in $Q^\circ$, one obtains a shortest $Q$-Minkowski billiard orbit in $Q^\circ$. The common boundary $\partial \beta Q^\circ \cap \partial K$ is origin-symmetric, so connecting any two of its antipodal points yields a shortest $Q$-Minkowski billiard orbit in $K$, with $Q$-length $4\beta$. Therefore, $\vol K \cdot \vol Q \ge \beta^n \vol Q^\circ \cdot \vol Q \overset{\text{Mahler}}\ge  \frac{(4\beta)^n}{n!} = \frac{c(K \times Q)^n}{n!}$.} A middle-ground between the more general (and false) and the less general (and equivalent to Mahler's) versions could be the case where only $K$ is required to be origin-symmetric. Reformulating is as a worm problem, we obtain the problem of the smallest \emph{origin-symmetric} $Q$-cover. Let $K$ be an origin-symmetric $Q$-cover. Let us say that a triple of outward normal directions of $Q$ \emph{encloses the origin} if their convex hull contains the origin in its interior. For every triple of outward normal directions of $Q$ that encloses the origin, there are two corresponding $Q$-normal triangles, differing by a central symmetry. Let $T$ be one of them, shifted so that $T \subset K$. The set $\conv (T \cup (-T))$ may already be a tight hexagon associated with that triple of outward normal directions of $Q$. But even if not, we can shift $T$ and $-T$ towards each other, so that $\conv (T \cup (-T))$ strictly decreases, and eventually it becomes a tight hexagon. In other words, if $K$ is an origin-symmetric $Q$-cover, then it also covers a tight hexagon, centered at the origin, for each triple of outward normal directions of $Q$ enclosing the origin. The other way around, if $K$ covers a tight centered hexagon for each triple of outward normal directions of $Q$ enclosing the origin, then every $Q$-normal triangle fits into $K$ trivially. 
Therefore, we arrive at the following question, an affirmative answer to which is equivalent to \Cref{conj:viterbo} in the case when $n=2$ and $K$ is centrally symmetric.

\begin{question}\label{ques:half-symmetric-viterbo-covering}
Let $Q$ be a convex polygon. To each pair of parallel sides, assign the corresponding $Q$-normal line segment. To each triple of outward normal directions of $Q$ that encloses the origin, assign one tight hexagon, selected arbitrarily from the two-dimensional family of tight hexagons associated with this triple. Shift each assigned shape so that it is centered at the origin. Is it true that the convex hull of these shapes has area at least $\frac{1}{2 \area Q}$?
\end{question}

\section{A plunge: quadrilateral norm}\label{sec:quadrilateral-covering}

In this section we solve \Cref{prob:cover} in the case where $Q$ is a convex quadrilateral, thus proving the corresponding part of \Cref{thm:quadri-hexa-viterbo}.

\subsection{The Viterbo inequality for quadrilaterals}
Applying an appropriate affine transformation, we can assume that $Q$ is a quadrilateral with perpendicular diagonals of equal lengths. In the generic case there are two triples of sides of $Q$ whose outward normal directions enclose the origin in the sense that their convex hull contains the origin in its interior. In degenerate situations (opposite sides parallel) one or both triples may become pairs instead, and the corresponding situations are simpler to investigate, but they also follow from the generic case by ``passing to the limit'' type of argument. So we stick to the generic case, so that there are four $Q$-normal triangles. \Cref{prob:cover} calls for the search of a convex shape of smallest area into which all four normal triangles fit. In \Cref{fig:quadri-cover} we depict an optimal shape, which turns out to be a square (but sometimes it is not unique).

\begin{theorem}\label{thm:serbia}
Let $Q$ be a convex quadrilateral with perpendicular diagonals of unit length (so the area of $Q$ is $1/2$). Then any $Q$-cover has area at least $1$, and there is a $Q$-cover of area exactly $1$: the unit square with the sides parallel to the diagonals of $Q$.
\end{theorem}

\begin{proof}
Assume that $Q$ has no parallel sides (otherwise, approximate it by one with no parallel sides). Let $K_0$ be the unit square with the sides parallel to the diagonals of $Q$. 

\underline{Step 1}. We first show that $K_0$ is $Q$-cover. Denote the vertices of $Q$ by $A,B,C,D$, and circumscribe a square around $Q$ with sides parallel to $AC$ and $BD$. Let $E$ and $F$ be two opposite vertices of the circumscribed square, and suppose that lines $AB$ and $CD$ intersect at the point $G$, as in \Cref{fig:quadri-cover}. It is a corollary of the Pappus theorem from projective geometry that the points $E$, $F$, and $G$ lie on the same line. Namely, the dual Pappus theorem (which is equivalent to the Pappus theorem itself) can be applied to two triples of lines: three vertical lines $AE,BD,CF$, and three horizontal lines $BF,AC,DE$. The conclusion is that the lines $AB,CD,EF$ pass through the same point $G$. In other words, if we circumscribe a rectangle around $\triangle AGD$, with sides parallel to $AC$ and $BD$, then it will actually be a square. Two out of four $Q$-normal triangles are homothetic to $\triangle AGD$ rotated by $\pm \pi/2$. Therefore, we can inscribe two $Q$-normal triangles in the same square; see \Cref{fig:quadri-cover}, on the right, where the yellow and orange triangles are inscribes in the red square. We need to show that the side length of this red square is $1$. (The argument for two remaining $Q$-normal triangles, blue and green in the figure, is similar.)

\begin{figure}[ht]
    \centering
    \includegraphics[width=0.8\textwidth]{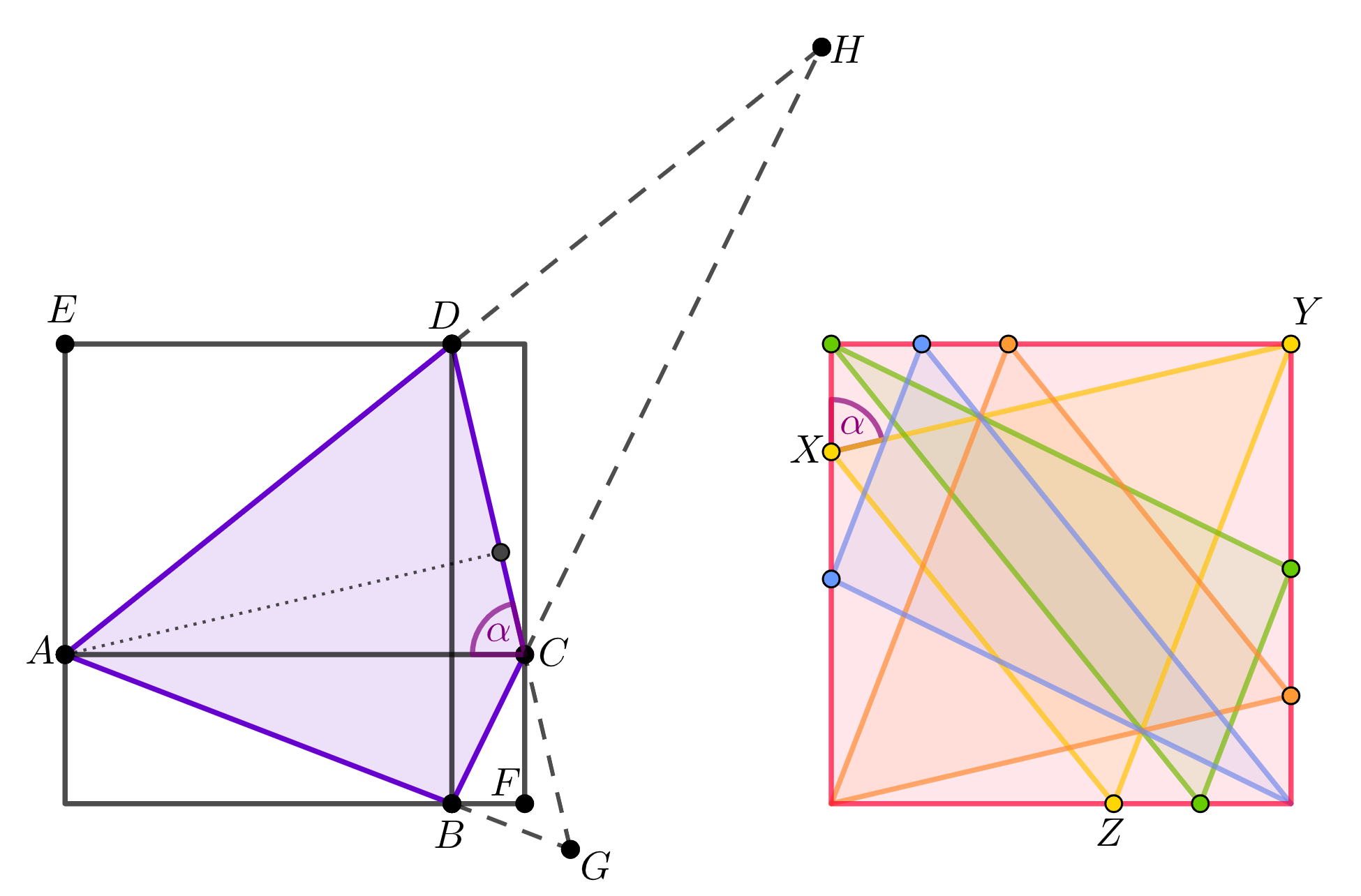}
    \caption{The quadrilateral $Q$ (purple, left) and the $Q$-covering square $K_0$ (red, right) with four $Q$-normal triangles inside.}
    \label{fig:quadri-cover}
\end{figure}

As it is explained in \Cref{ex:shift-origin}, for the computation of the $Q$-perimeter of the yellow triangle $\triangle XYZ$ we can ``shift the origin'' (in the plane of quadrilateral $Q$) to the point $A$. Then
\[
1 = \len_Q(X \to Y \to Z \to X) = \dist(A, CD) \cdot |XY|,
\]
where $\dist(A, CD)$ denotes the euclidean distance from $A$ to the line $CD$, and $|XY|$ denotes the euclidean length of the segment $XY$. Denote by $\alpha$ the angle $\measuredangle ACD$, and observe that the side length of the red square equals
\[
|XY| \sin \alpha = \frac{\sin \alpha}{\dist(A, CD)} = \frac1{|AC|} = 1.
\]

\underline{Step 2}. We show that any $Q$-cover has the area at least that of $K_0$. The main tool in the proof will be computing oriented areas enclosed by curves (\Cref{def:action}). 
We label the vertices of the four $Q$-normal triangles inside $K_0$ by $p_1, \ldots, p_{12}$ following the cyclic order along the boundary of $K_0$. Let $v_i(t) = p_i + t$, $t \in \mathbb{R}^2$, be sliding copies of the vertices $p_i$. Denote by $T_y(t_y) = \triangle v_1 v_5 v_8 = t_y + \triangle p_1 p_5 p_8$ the sliding yellow triangle (it received indices $1$, $5$, $8$ depending on this particular shape of $Q$, but indices do not matter in the following argument). Similarly, let $T_o(t_o) = \triangle v_2 v_7 v_{11}$ be the sliding orange triangle, $T_b(t_b) = \triangle v_3 v_6 v_{10}$ the sliding blue triangle, and $T_g(t_g) = \triangle v_4 v_9 v_{12}$ the sliding green triangle. Now consider the following closed polygonal curves:
\begin{align*}
\Gamma_{AD} &= v_1 \to v_2 \to v_3 \to v_4 \to v_5 \to v_7 \to v_8 \to v_9 \to v_{10} \to v_{11} \to v_1, \\
\Gamma_{AB} &=  v_1 \to v_2 \to v_3 \to v_4 \to v_6 \to v_7 \to v_8 \to v_9 \to v_{10} \to v_{12} \to v_1,\\
\Gamma_{BC} &= v_1 \to v_4 \to v_6 \to v_7 \to v_{10} \to v_{12} \to v_1, \\
\Gamma_{CD} &= v_1 \to v_4 \to v_5 \to v_7 \to v_{10} \to v_{11} \to v_1.
\end{align*}
They are obtained in the following way. To construct $\Gamma_{AD}$, consider all lines $\ell$ perpendicular to $AD$ such that all intersections of $\ell$ with the boundary of $K_0$ are among the points $p_i$. Note down these $p_i$, observe their cyclic order along the boundary of $K_0$, and construct $\Gamma_{AD}$ out of the corresponding points $v_i$ in the same order. The curves $\Gamma_{AB}$, $\Gamma_{BC}$, and $\Gamma_{CD}$ are constructed similarly. We found these curves when examining the billiard properties of the configuration $(K_0,Q)$. It turns out that it enjoys some kind of Zoll property: for almost every initial position of a billiard ball in $K_0$ and almost every initial momentum in $\partial Q$, the resulting trajectory is a $4$-periodic billiard orbit of fixed length. The four $Q$-normal triangles arise as limiting $3$-periodic billiard orbits, and the curves $\Gamma_\bullet$ appear naturally when analyzing the space of these $3$- and $4$-periodic orbits. 

\begin{figure}[ht]
    \centering
    \includegraphics[width=0.8\textwidth]{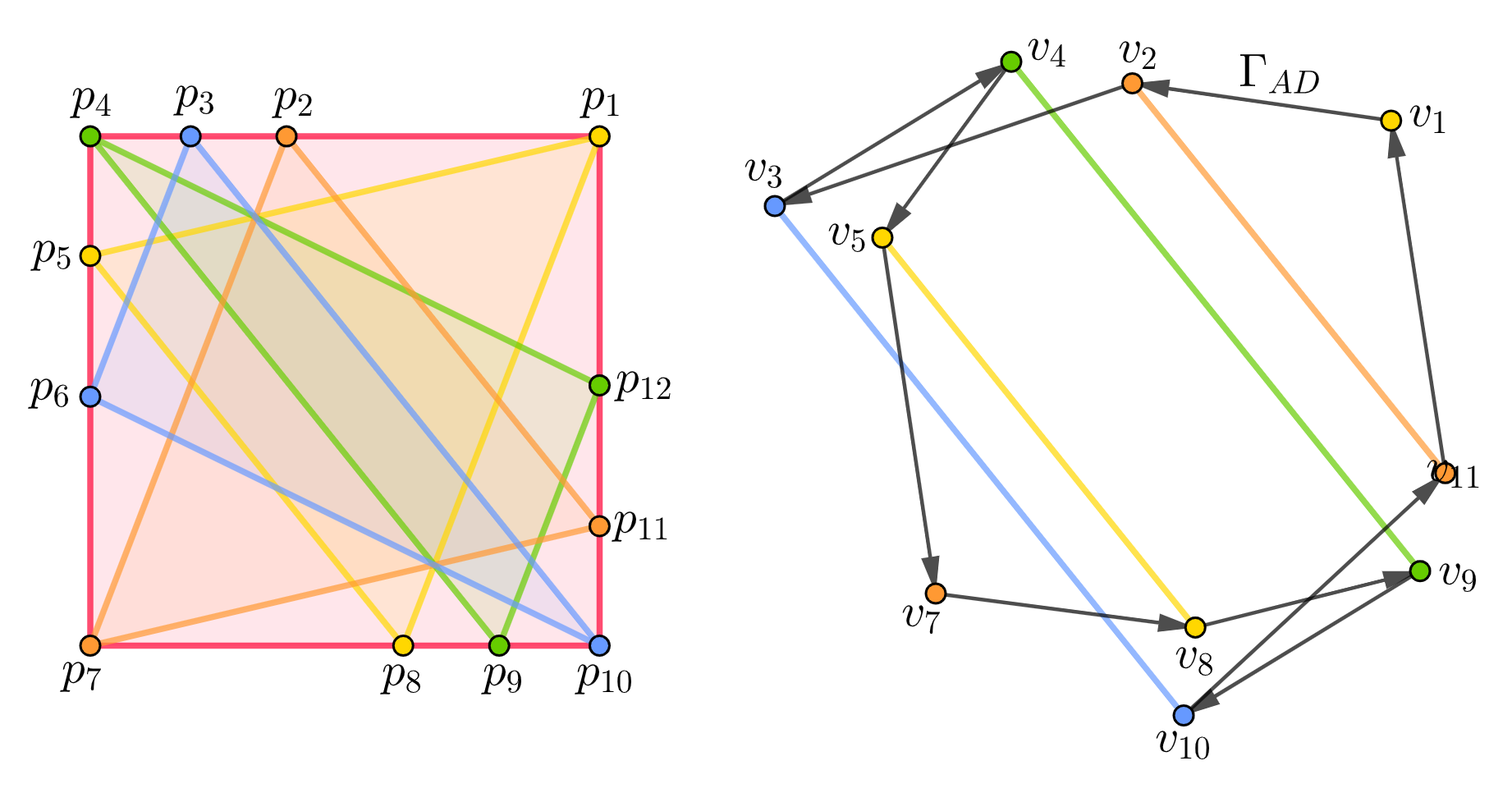}
    \caption{Labeling of the vertices of the $Q$-normal triangles and the contour $\Gamma_{AD}$.}
    \label{fig:quadri-contour}
\end{figure}

Each of the four curves $\Gamma_\bullet$ depends on the slide parameters $t = (t_y,t_o,t_b,t_g) \in (\mathbb{R}^2)^4$, and each of the oriented areas $\act(\Gamma_\bullet)$ is a linear function of $t$. We will find a convex combination 
\[
\Theta_\alpha(t) = \alpha_{AD} \act(\Gamma_{AD}) + \alpha_{AB} \act(\Gamma_{AB}) + \alpha_{BC} \act(\Gamma_{BC}) + \alpha_{CD} \act(\Gamma_{CD})
\]
that is constant. For $t=0$ each of the curves $\Gamma_\bullet$ traces the boundary of $K_0$ counterclockwise, so the constant actually equals $\Theta_\alpha(0) = \area(K_0) = 1$.

Introduce vectors $u,w \in \mathbb{R}^2$ and numbers $a,b > 1$ in the following way:
\begin{align*}
    u &= v_2 - v_{11} = v_5 - v_8, \\
    au &= v_3 - v_{10} = v_4 - v_9, \\
    w &= v_6 - v_3 = v_9 - v_{12}, \\
    bw &= v_8 - v_1 = v_7 - v_2.
\end{align*}

Now we compute the variations of oriented areas with respect to sliding $T_y$:
\begin{align*}
    \act(\Gamma_{AD}(t_y+s,t_o,t_b,t_g)) - \act(\Gamma_{AD}(t_y,t_o,t_b,t_g)) &= \frac12 s \wedge (v_2 - v_{11} + v_9 - v_4) = \frac12 s \wedge (1-a)u, \\
    \act(\Gamma_{AB}(t_y+s,t_o,t_b,t_g)) - \act(\Gamma_{AB}(t_y,t_o,t_b,t_g)) &= \frac12 s \wedge (v_9 - v_7 + v_2 - v_{12}) = \frac12 s \wedge (1-b)w, \\
    \act(\Gamma_{BC}(t_y+s,t_o,t_b,t_g)) - \act(\Gamma_{BC}(t_y,t_o,t_b,t_g)) &= \frac12 s \wedge (v_4-v_{12}) = \frac12 s \wedge (au+w), \\
    \act(\Gamma_{CD}(t_y+s,t_o,t_b,t_g)) - \act(\Gamma_{CD}(t_y,t_o,t_b,t_g)) &= \frac12 s \wedge (v_7 - v_{11}) = \frac12 s \wedge (u+bw).
\end{align*}

If we choose $\alpha_\bullet$ in a way that guarantees
\[
\alpha_{AD} (1-a)u + \alpha_{AB}(1-b)w + \alpha_{BC}(au+w) + \alpha_{CD}(u+bw) = 0,
\]
the resulting function $\Theta_\alpha(t)$ will not depend on $t_y$. It turns out that the conditions that $\Theta_\alpha(t)$ does not depend on $t_o,t_b,t_g$ are exactly the same. This can be verified directly, but we also indicate another way of seeing this. Denote by $\delta_{t_\bullet} \Theta_\alpha$ the expression obtained by varying $t_\bullet$. Then $\delta_{t_o} \Theta_\alpha = -\delta_{t_y} \Theta_\alpha$, $\delta_{t_b} \Theta_\alpha = -\delta_{t_g} \Theta_\alpha$ just due to the central symmetry of the configuration. The fact that $\delta_{t_b} \Theta_\alpha = - \delta_{t_y} \Theta_\alpha$ follows from the observation that, in all four curves, the union of yellow and blue vertices alternates with the union of orange and green vertices. Consequently, in $\delta_{t_b} \Theta_\alpha$ and $\delta_{t_y} \Theta_\alpha$ we will see---with opposite signs---only the terms corresponding to those vertices whose two neighbors are one blue and one yellow.

All non-negative solutions of the zero-variation equation are parametrized by arbitrary $x,y \ge 0$ as follows: 
\begin{align*}
\alpha_{AD} &= \frac{xa}{a-1} + \frac{y}{a-1}, \\ 
\alpha_{AB} &= \frac{x}{b-1} + \frac{yb}{b-1}, \\ 
\alpha_{BC} &= x, \\ 
\alpha_{CD} &= y.
\end{align*}
Fix any non-negative solution with $\alpha_{AD} + \alpha_{AB} + \alpha_{BC} + \alpha_{CD} = 1$. \Cref{lem:croissant} shows that $\act(\Gamma_{\bullet}) \le \area(\conv(\Gamma_{\bullet}))$ for any of the four curves we considered. Note that the lemma is indeed applicable; for example, the assumptions of the lemma for $\Gamma_{AD}$ boil down to the fact that the arrows $p_1 \to p_1$, $p_2 \to p_{11}$, $p_3 \to p_{10}$, $p_4 \to p_9$, $p_5 \to p_8$, $p_7 \to p_7$ (two of which are trivial) have the same direction, and their (euclidean) lengths form a unimodal sequence. This can be seen from \Cref{fig:quadri-contour}, where these vectors are consecutive slices of $K_0$; the unimodality follows from the convexity of $K_0$.

We are ready to conclude the proof. Suppose $K$ is a $Q$-cover, that is, for some choice of $t=(t_y,t_o,t_b,t_g)$, we have $T_y, T_o, T_b, T_g \subset K$. Then $\Gamma_{AD}, \Gamma_{AB}, \Gamma_{BC}, \Gamma_{CD} \subset K$. With the choice of $\alpha_\bullet$ as above, we see that
\begin{align*}
\area(K) &\ge \alpha_{AD} \cdot \area(\conv(\Gamma_{AD})) + \mbox{ three similar summands} \\
&\ge \alpha_{AD} \cdot \act(\Gamma_{AD }) + \mbox{ three similar summands} \\
&= \Theta_\alpha(t) = \Theta_\alpha(0) = \area(K_0).
\end{align*}
\end{proof}

As a corollary of \Cref{thm:serbia}, we see that \Cref{thm:quadri-hexa-viterbo} holds in the case where $Q$ is a quadrilateral.

The problem of engineer Q.~Uadrilateral from the introduction is equivalent via an affine change to the problem we just solved, with the same parameters $a, b$. The smallest area of a $Q$-cover in that problem can be computed directly to be $\frac{ab(ab-1)}{2ab-a-b}$.

\subsection{Characterization of equality cases}

To characterize the cases where $\area K = \frac{ab(ab-1)}{2ab-a-b}$, one needs to describe carefully all situations when all four curves $\Gamma_{AD}$, $\Gamma_{AB}$, $\Gamma_{BC}$, $\Gamma_{CD}$ enclose the same area of the same shape $K$. It does not necessarily mean that each $\Gamma_\bullet$ strictly follows the boundary of $\partial K$ counterclockwise; there might be parts of $\Gamma_\bullet$ that trace a certain arc back and forth without any contribution to the enclosed area. However, 1) these parts must lie inside $K$, otherwise we have a strict inequality in \Cref{lem:croissant}, and 2) if we remove those parts, the simplified $\Gamma_\bullet$ does follow the boundary of $\partial K$ counterclockwise. The case analysis behind the classification of equality cases is straightforward but tedious. We only formulate the result below.

It is convenient to parametrize all quadrilaterals $Q$ up to affine equivalence using the parameters $a$ and $b$; we can assume that $a \le b$. An elementary computation shows that those parameters can be seen in \Cref{fig:quadri-cover} as $a = \frac{|AG|}{|AB|}$ and $b=\frac{|AH|}{|AD|}$. 
We also need to recall the cases when $Q$ has one or two pairs of parallel sides; they can be thought as corresponding to one or both parameters $a$ and $b$ equal to $+\infty$. For example, if $b = +\infty$, it corresponds to the case when the blue and green triangles become so thin so that degenerate to a single line segment, so we can assume $v_3 = v_4 = v_6$ and $v_9 = v_{10} = v_{12}$. In this case $Q$ is a trapezoid. 
If both $a$ and $b$ equal $+\infty$, then $Q$ is a parallelogram, and the covering problem becomes completely trivial, as it asks for the smallest convex hull of two non-parallel segments. Another type of degeneration can be considered if we allow $a=1$: in this case $Q$ becomes a triangle regardless of the value of $b$, and this problem was already solved in \Cref{sec:triangular-covering}.

We now collect the description of all equality cases depending on $1 \le a \le b \le +\infty$, from the least generic case to the most generic. We assume that $Q$ is affinely reduced to the form with perpendicular diagonals of equal lengths, and whenever we refer to $K_0$, it is the square from \Cref{thm:serbia}.
\begin{itemize}
    \item $a=1$, the value of $b$ does not matter. $Q$ is a triangle. The space of equality cases is two-dimensional and was described in \Cref{sec:triangular-covering}. 
    \item $a, b = +\infty$. $Q$ is a parallelogram. The space of equality cases is two-dimensional and described as follows: there are two $Q$-normal line segments, and the area of their convex hull is minimal whenever they intersect. 
    \item $a \in (1, +\infty)$, $b = +\infty$. $Q$ is a trapezoid. The only equality case is $K_0$.\footnote{A discussion of this case (as well the corresponding special case of \Cref{thm:serbia}) appeared in~\cite{rudolf2022viterbo} following correspondence with the first‑named author, in which these considerations were raised.}
    \item $a,b \in (1, +\infty)$, $(a-1)(b-1) = 1$. $Q$ is a quadrilateral without parallel sides but it satisfies a constraint that enforces $p_5 = p_6$ and $p_{11} = p_{12}$ in \Cref{fig:quadri-contour}. The space of equality cases is one-dimensional: one endpoint is $K_0$, and it is possible to slide the $Q$-normal triangles to obtain a one-parametric family of $Q$-covering hexagons of the same area (see \Cref{fig:quadri3}). These equality cases (except for the boundary one with $K_0$) have not been described before in the literature. 
    \item $(a-1)(b-1) \neq 1$. $Q$ is a maximally generic quadrilateral. The only equality case is $K_0$.
\end{itemize}

\begin{figure}[ht]
    \centering
    \includegraphics[width=0.8\linewidth]{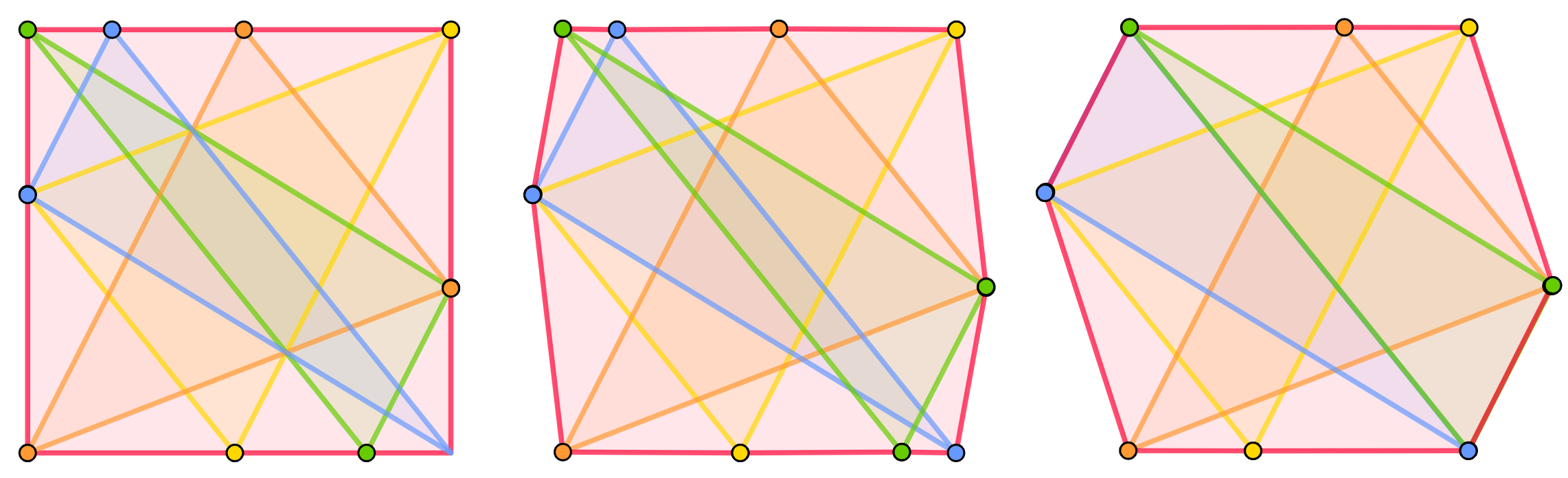}
    \caption{A one-parameter family of smallest $Q$-covers in the case where $Q$ is a quadrilateral with $(a-1)(b-1)=1$.}
    \label{fig:quadri3}
\end{figure}

For many of these equality cases it can be shown that the interior of $K \times Q$ is symplectomorphic to a ball. The idea is to construct explicit symplectic embeddings (see \cite{traynor1995symplectic,schlenk2005embedding}), and then in dimension $4$ the existence of symplectomorphisms follows from the heavy machinery of~\cite{latschev2013gromov}. The details of these embeddings are given in~\cite{rudolf2022viterbo} for most equality cases, except for the new ones with $(a-1)(b-1) = 1$. We leave it as an open question whether the new equality cases with $(a-1)(b-1) = 1$ are symplectic balls.

\subsection{Weak Viterbo inequality via approximation by quadrilaterals}
Here we quote an observation of Y.~Ostrover (private communication). Having proved the Viterbo inequality for quadrilaterals (\Cref{thm:serbia}), we can establish a non-sharp version of the Viterbo inequality for all convex shapes $K$ and $Q$ by approximating $Q$ with a quadrilateral. A result of W.~Kuperberg~\cite{kuperberg1983minimum} states that around any convex shape $Q$ one can circumscribe a quadrilateral $Q'$ such that $\frac{\area Q'}{\area Q} < \kappa$, where $1 < \kappa < \sqrt{2}$ is a fixed non-explicit constant, later specified in~\cite{fodor2026minimal} as $\kappa = (1 - 2.6 \cdot 10^{-7})\sqrt{2}$.\footnote{It is also conjectured in~\cite{kuperberg1983minimum} that the worst shape $Q$ is a regular pentagon, and that one may take $\kappa=\frac{3}{\sqrt{5}}$.} Then we have:
\[
\area K \cdot \area Q \ge \area K \cdot \frac{\area Q'}{\kappa} \ge \frac1{2\kappa} c(K \times Q')^2 \ge \frac1{2\kappa} c(K \times Q)^2.
\]
Here we applied \Cref{thm:serbia} to the pair $(K,Q')$, and used the monotonicity of capacities $c(K \times Q) \le c(K \times Q')$ for $Q \subset Q'$. This monotonicity is known in general for convex sets in $\mathbb{R}^2 \times \mathbb{R}^2$, but it can also be explained elementarily. As discussed in \Cref{sec:history}, the capacity is symmetric in $K$ and $Q$; and by \Cref{thm:capacity-bezdek} we have $c(Q \times K) \le c(Q' \times K)$. This proves that $\kappa < \sqrt{2}$ in \Cref{ques:constant-viterbo}.

\section{A revelation: pentagonal norm}\label{sec:pentagonal-covering}

In this section we consider the case where $Q$ is a convex pentagon. The space of pentagons is four-dimensional, if we consider them up to affine equivalence.
We have no complete solution of \Cref{prob:cover} for any pentagon, but we consider two scenarios that we find enlightening. The first of them leads to an example due to P.~Haim-Kislev and Y.~Ostrover, which disproves \Cref{conj:viterbo}. The framework of the worm problem for $Q$-normal triangles provides a simple explanation of that example.

\subsection{Regular pentagon and the counterexample of Haim-Kislev and Ostrover}
Let $Q$ be a regular pentagon. There are ten $Q$-normal triangles, all congruent with angles $\pi/5$, $\pi/5$, $3\pi/5$; they are obtained by rotating a single one by angles $k\pi/5$, $k = 0, 1,\ldots,9$. Let $K$ be an appropriately rescaled copy of $Q$, rotated by $\pi/2$, such that every $Q$-normal triangle fits (tightly) into $K$. The paper~\cite{haim2026counterexample} used the Minkowski billiard framework to compute $c(K \times Q)$. The machinery of \Cref{sec:reduction} allows us to significantly simplify their argument: since every $Q$-normal triangle fits into $K$, whereas no larger homothetic copy of any $Q$-normal triangle does, \Cref{thm:normal-cover} immediately gives $c(K \times Q) = 1$. A direct computation of areas shows that
\[
\area K \cdot \area Q = \underbrace{\frac{5}{\sqrt{5}+3}}_{<1} \cdot\frac{c(K\times Q)^2}{2},
\]
which violates the Viterbo inequality (\Cref{conj:viterbo}).

\begin{figure}[ht]
    \centering
    \includegraphics[width=0.5\linewidth]{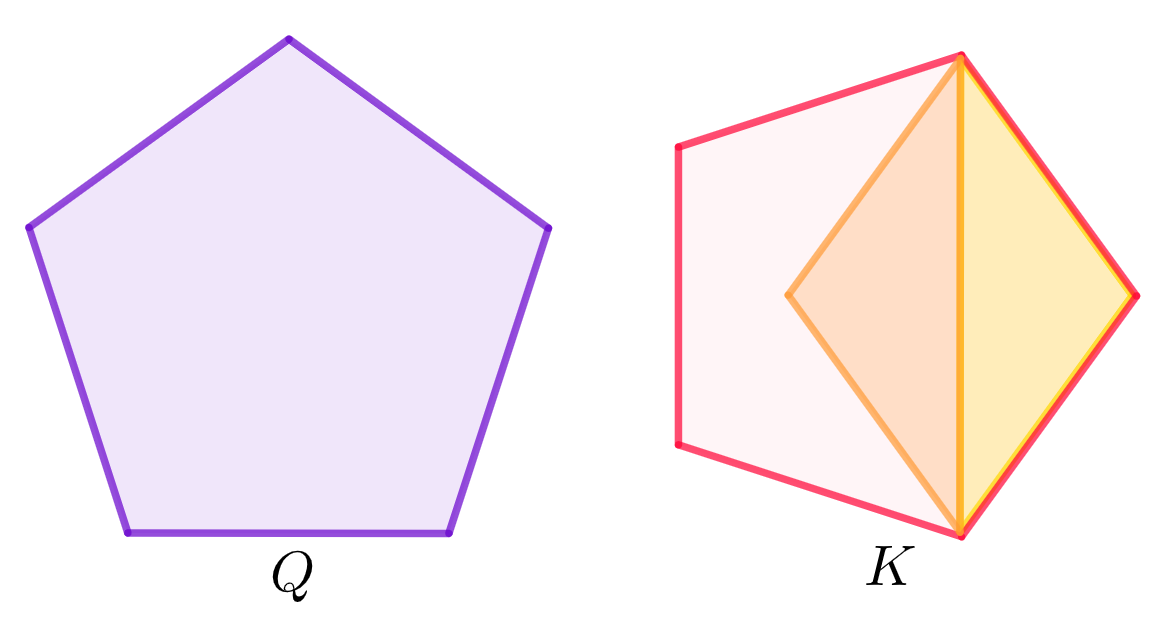}
    \caption{The Haim-Kislev--Ostrover example. Two of the ten $Q$-normal triangles (yellow and orange) are shown covered by $K$; the remaining eight are obtained by applying the symmetries of $K$.}
    \label{fig:penta1}
\end{figure}

It seems likely that this $K$ has the smallest area among all $Q$-covers, but we do not know how to prove this.

\subsection{Squares truncated at one corner}
Let $Q$ be a pentagon obtained by cutting off one corner of a square along a line segment. These pentagons form a two-dimensional subfamily within the full four-dimensional family of pentagons. For such $Q$, there are four $Q$-normal triangles, two of which are degenerate.
In the fall of 2023, the first-named author, together with Mingkun Liu, invited Lou Meylender and Thomas Thalmaier, two undergraduate students at the University of Luxembourg, to investigate this problem experimentally within the framework of the Experimental Mathematics Lab. They implemented an optimization program searching for the smallest $Q$-cover. For the shapes they found \Cref{conj:viterbo} holds with a margin. We formulate their findings as a conjecture about the optimal shape of $Q$-covers.

\begin{figure}[ht]
    \centering
    \includegraphics[width=0.9\linewidth]{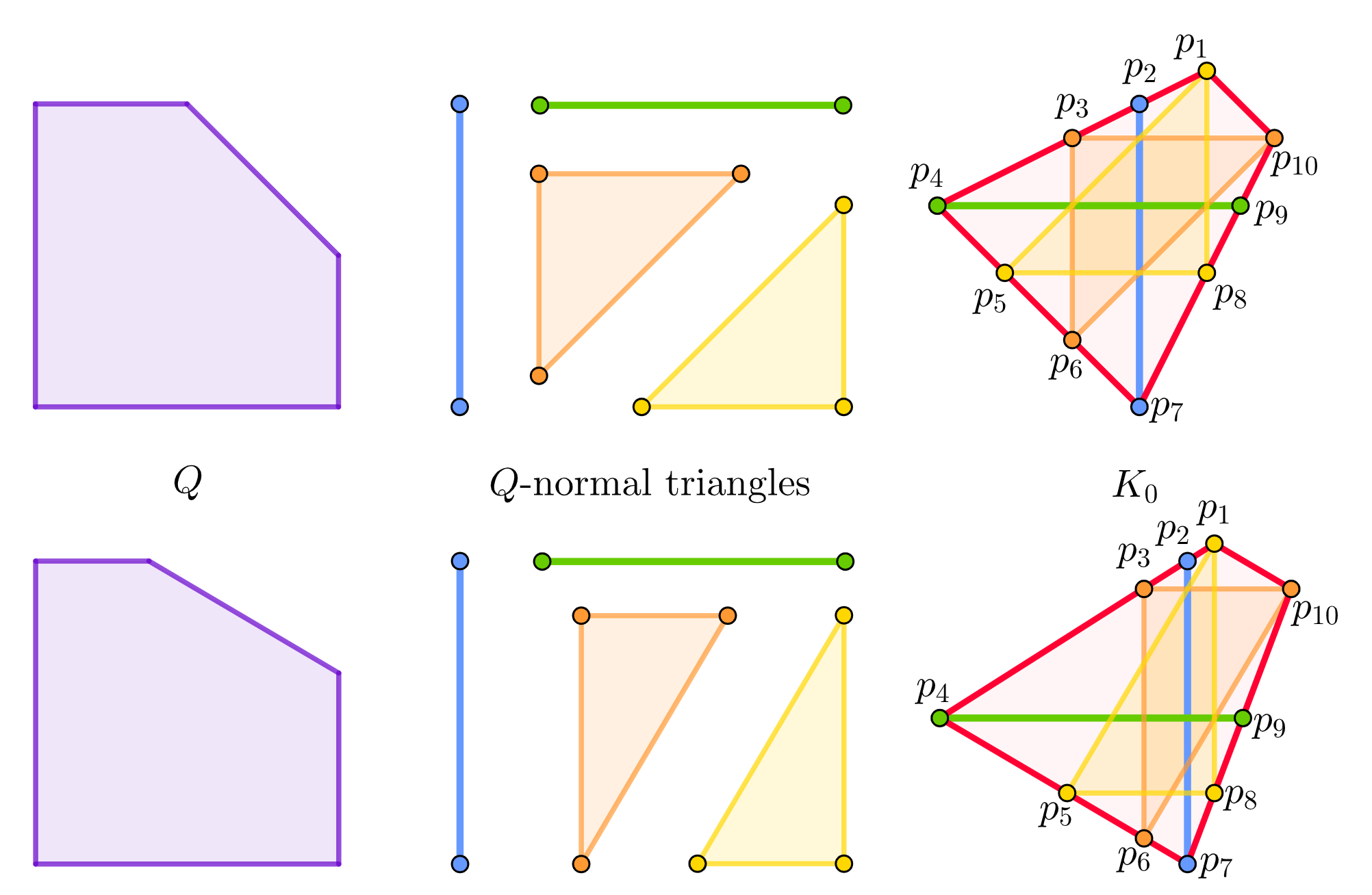}
    \caption{The Meylender--Thalmaier conjecturally optimal $Q$-cover $K_0$ (right) in the case where $Q$ is a square truncated at one corner.}
    \label{fig:penta2}
\end{figure}

Label the vertices of the four $Q$-normal triangles $p_1, \ldots, p_{10}$ as shown in \Cref{fig:penta2}. The trapezoid $K_0$ in this figure can be described as follows: it corresponds to a unique placement (up to translation) of the four $Q$-normal triangles for which the points $p_1, p_2, p_3, p_4$ are collinear, the points are $p_4, p_5, p_6, p_7$ are collinear, and the points $p_7, p_8, p_9, p_{10}$ are collinear. 

\begin{conjecture}\label{conj:meylender-thalmaier}
Let $Q$ be a pentagon obtained from a square by truncating one of its corners. Then the systolic ratio $\frac{c(K\times Q)^2}{2 \area K \cdot \area Q}$ is strictly less than $1$ for any convex shape $K$, and it is maximized with respect to $K$ if and only if $K$ is a translate of $K_0$ or $-K_0$.
\end{conjecture}

Note that this conjectured inequality is strictly stronger than the corresponding special case of the Viterbo inequality. To quantify how much it is stronger, let us parametrize the space of squares truncated at one corner as follows. Let $A,B,C,D,E$ be the vertices of $Q$ so that $AB \parallel DE$ and $AE \parallel BC$. Let the lines $AB$ and $CD$ intersect at the point $F$, and let the lines $CD$ and $AE$ intersect at $G$. Let $\alpha = \frac{|BF|}{|AF|}$, $\beta = \frac{|EG|}{|AG|}$, and $\gamma = 1 - \alpha - \beta$. The parameters satisfy the constraints $0 < \alpha, \beta, \gamma < 1$. An elementary computation allows one to express the areas of $Q$ and $K_0$ in terms of $\alpha, \beta, \gamma$, showing that
\[
\frac{c(K_0\times Q)^2}{2 \area K_0 \cdot \area Q} = \frac{1}{ 1 + \frac{(\alpha^2 + \beta^2) \gamma^2}{1 - \alpha^2 - \beta^2 - \gamma^2}} < 1. 
\]

\section{A glimmer: hexagonal norm}\label{sec:hexagonal-covering}

In this section we study \Cref{prob:cover} in the case where $Q$ is a convex centrally symmetric hexagon. We solve it completely in the case where $Q$ is affinely equivalent to a regular hexagon; this proves the remaining part of \Cref{thm:quadri-hexa-viterbo}.

\subsection{Regular hexagon}
Let $Q$ be a hexagon affinely equivalent to a regular one. After an affine transformation, we can assume that $Q$ is the regular hexagon of unit euclidean width. The $Q$-normal triangles can have sides of three parallel classes, orthogonal to the sides of $Q$. This list of $Q$-normal triangles consists of two regular triangles with side length $2/3$, and three unit line segments (see \Cref{fig:hexa1}). The regular triangle $K_0$ whose three altitudes are $Q$-normal line segments is expected to be an area-minimizing $Q$-cover, since we already know that the pair $(Q, K_0)$ attains equality in the Viterbo inequality. We will confirm this, but we will also discover that there are many other area-minimizing $Q$-covers different from $K_0$ and $-K_0$, and moreover, the set of minimizers consists of components of different dimensions. This adds to the complexity of this $Q$-cover problem.

\begin{figure}[ht]
    \centering
    \includegraphics[width=\linewidth]{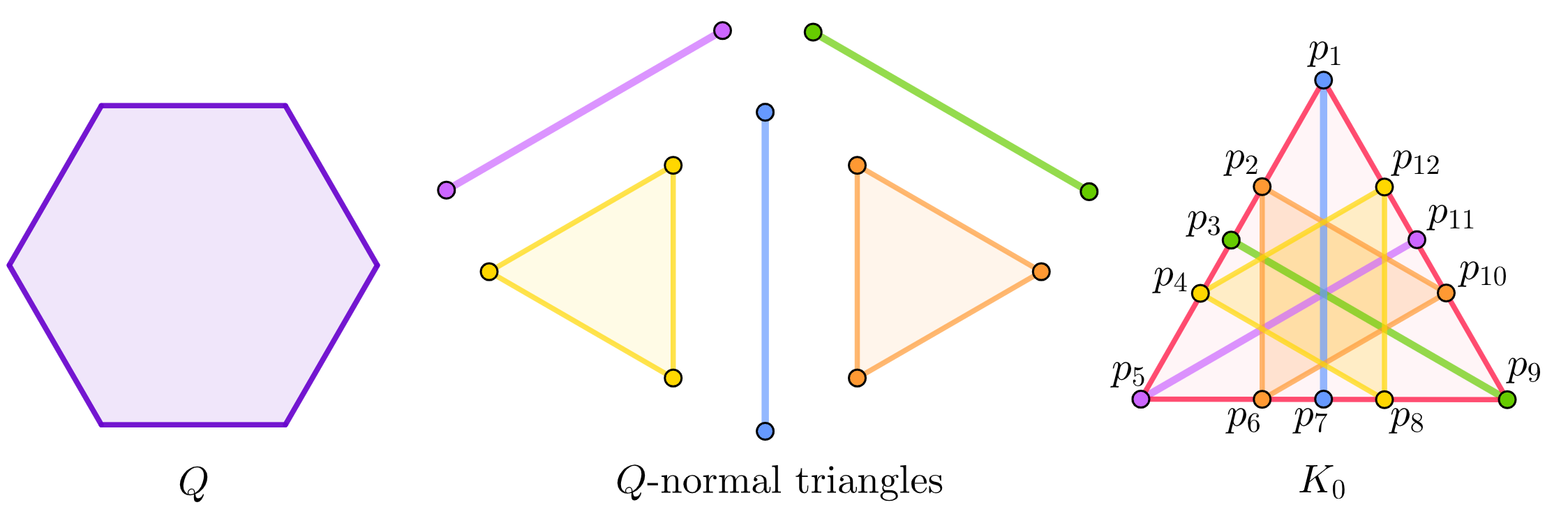}
    \caption{When $Q$ is a regular hexagon, the regular triangle $K_0$ is one of the smallest $Q$-covers.}
    \label{fig:hexa1}
\end{figure}

To solve this problem we employ a new trick, whose essence is captured in the following lemma. 
\begin{lemma}\label{lem:mitrofanov-cap}
Let $\Phi \subset K$ be nested convex polygons. For each side $s_i$ of $\Phi$, let $v_i$ be any vertex of $K$ that maximizes the inner product $\langle \cdot, \nu_i\rangle$ with the outward normal $\nu_i$ to $s_i$. The (possibly degenerate) triangle $C_i \coloneq \conv(\{v_i\} \cup s_i)$ will be called the \emph{cap} based on $s_i$. Then the caps $C_i$ based on all sides of $\Phi$ do not overlap, meaning that their interiors are pairwise disjoint.
\end{lemma}

\begin{proof}
In this argument, disregard all caps $C_i$ that have empty interior. For each $C_i$ with non-empty interior, introduce the linear function $f_i : \mathbb{R}^2 \to \mathbb{R}$ such that $f_i(v_i) = 1$ and $f_i\vert_{s_i} = 0$. By construction, $f_i\vert_\Phi \le 0$ and $f_i\vert_F \le 1$. Then for any $j \neq i$, we have $(f_i - f_j)\vert_{s_i} = - f_j\vert_{s_i} \ge 0$, and $f_i(v_i) - f_j(v_i) = 1 - f_j(v_i) \ge 0$. Therefore, $f_i \ge f_j$ on $C_i$, and moreover, $f_i > f_j$ on the interior of $C_i$. If there was a point $x$ in the intersection of the interiors of $C_i$ and $C_j$, we would have $f_i(x) > f_j(x)$ and $f_j(x) > f_i(x)$ simultaneously. Therefore, $C_i$ and $C_j$ do not overlap.
\end{proof}

\begin{theorem}\label{thm:serbia2}
Let $Q$ be the regular hexagon of unit width, and let $K$ be any $Q$-cover. Then $\area K \ge \area K_0 = \frac{1}{2 \area Q}$, where $K_0$ is the regular triangle with unit height and sides parallel to those of $Q$. 
\end{theorem}

\begin{proof}
Let $T_y, T_o, L_1, L_2, L_3 \subset K$ be the five $Q$-normal shapes shifted to fit into $K$. In fact, we can assume that $K$ is the convex hull of these five shapes, and so it is a convex polygon. Denote by $A$ the area of $T_y$. Then $\area K_0 = 3A$, and we want to prove that $\area K \ge 3A$.

Let us first consider the relative position of $T_y$ and $T_o$. Similarly to the argument at the end of \Cref{sec:triangular-covering}, we check whether $\Phi \coloneq \conv (T_y \cup T_o)$ is a tight hexagon, that is, an area-minimizing translation cover for $T_y$ and $T_o$ (its area is $2A$, according to \Cref{sec:triangular-covering}). If not, we can shift $T_y$ and $T_o$ towards each other, so that $\Phi$ strictly decreases until it becomes a tight hexagon. Then shift $K$ to make $\Phi$ centered at the origin.

Now we apply \Cref{lem:mitrofanov-cap} to $\Phi \subset K$. We obtain six non-overlapping caps $C_i = \conv(\{v_i\} \cup s_i)$, $i=1,\ldots, 6$, based on the six sides $s_i$ of $\Phi$ (some of them may be degenerate). It suffices to prove that $\sum\limits_{i=1}^6 \area C_i \ge A$. To this end, let $R_i \coloneq \conv(\{0\} \cup s_i)$ be the triangles slicing $\Phi$. Let us prove that $\area C_1 + \area C_4 \ge \area R_1 = \area R_4$. If $s_1$ is a degenerate side, then $\area R_1 = 0$ and there is nothing to prove. Otherwise, $s_1$ and $s_4$ are parallel line segments. Let $f_1$ be the linear function such that $f\vert_{s_1} = -1$ and $f\vert_{s_4} = 1$. If we denote $\area R_1 = \area R_4 = A_1$, then we have $\area C_1 = A_1 (f\vert_{s_1}  -f(v_1)) = A_1 (-1 -f(v_1))$ and $\area C_4 = A_1(f(v_4) - f\vert_{s_4}) = A_1(f(v_4)-1)$. Since there is a line segment $L_j = [u,w] \subset K$ such that $f(u) - f(w) \ge \frac32(f\vert_{s_4} - f\vert_{s_1}) = 3$, it follows from the definition of the points $v_i$ that $f(v_4) - f(v_1) \ge 3$ as well. Therefore, $\area C_1 + \area C_4 \ge \area R_1 = \area R_4 = A_1 (f(v_4) - f(v_1) - 2) \ge A_1$. Summing this inequality with the analogous ones for the other caps, we obtain $\sum\limits_{i=1}^6 \area C_i \ge \frac12 \area \Phi = A$, which completes the proof.
\end{proof}

As a corollary we immediately obtain \Cref{thm:quadri-hexa-viterbo} in the case where $Q$ is an affinely regular hexagon.

An attempt to solve the same problem using the method of closed curves enclosing constant area leads to the following observations. First of all, such a curve exists: if we label the vertices of the $Q$-normal triangles along $\partial K_0$ as $p_1, \ldots, p_{12}$, and denote by $v_i$ the sliding copy of $p_i$, then $\Gamma = (v_1 \to \ldots \to v_{12} \to v_1)$ is a curve enclosing constant area, equal to $\area K_0$. It is the second step of the method that we do not know how to complete: why must $\act(\Gamma) \le \area \conv \Gamma$ hold?

However, the characterization of equality cases is easier done with this method rather than by analyzing the proof we gave. Recall that the smallest area of a $Q$-cover is attained when, roughly speaking, $\Gamma$ follows the boundary of $\partial K$ counterclockwise, except for possible portions where $\Gamma$ traces an arc inside $K$ back and forth without any contribution to the enclosed area. In \Cref{fig:hexa2} one sees an example where this cancellation happens thanks to the fact that $p_1 = p_{11}$, $p_3 = p_5$, and $p_7 = p_9$, and the exact position of the yellow triangle does not matter. Shifting the orange triangle does not change the enclosed area, so we obtain a two-dimensional family of equality cases. They are mostly hexagons (in \Cref{fig:hexa2} it is the most symmetric hexagon in the family), but the boundary cases are pentagons and quadrilaterals. The centrally symmetric image of this family yields another two-dimensional family of equality cases. To the best of our knowledge, all these equality cases were previously unknown. Together with $K_0$ and $-K_0$, these two families exhaust all equality cases, as can be verified by analyzing $\Gamma$.


\begin{figure}[ht]
  \centering
    \includegraphics[width=0.5\linewidth]{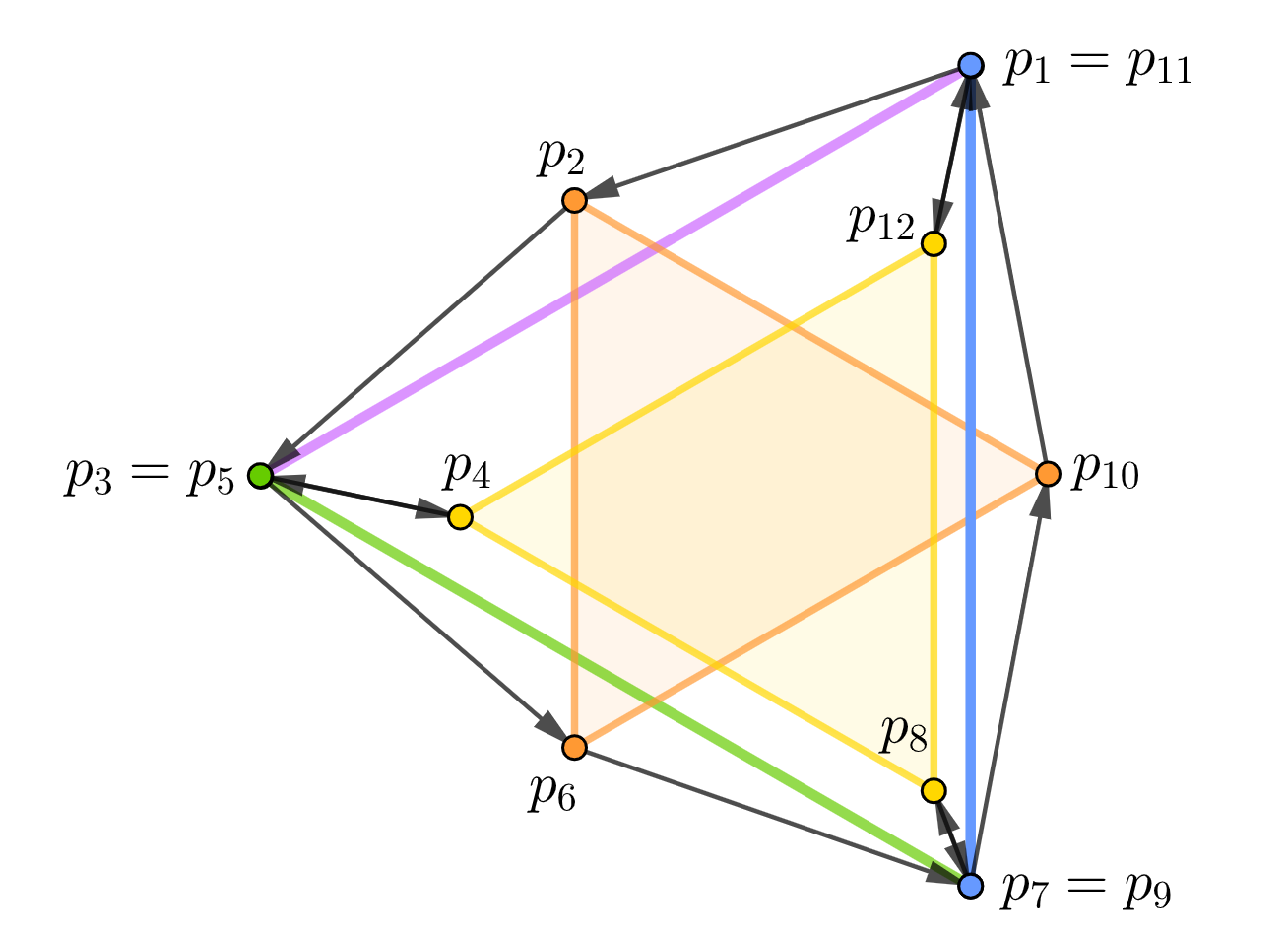}
  \caption{A two-dimensional family of smallest $Q$-covers: the orange triangle can be shifted---while preserving the area of the convex hull---as long as it intersects all three $Q$-normal line segments. The position of the yellow triangle is irrelevant as long as it stays inside the convex hull of the other four $Q$-normal shapes.} 
  \label{fig:hexa2}
\end{figure}

It is an open question whether each of these sets $K$ gives rise to a symplectic ball $K \times Q$.

\subsection{Non-regular hexagons}

The methods of this paper fail to provably find the smallest $Q$-cover for any other hexagon $Q$ other than the affinely regular one. We conclude by the following observation, obtained by swapping the role of $K$ and $Q$ in some of the equality cases that we already found earlier, and where $K$ turned out to be a centrally symmetric hexagon. There is a full-dimensional open set in the space of centrally symmetric hexagons such that, for any $Q$ from this set, there are at least four shapes $K$ attaining equality in the Viterbo inequality. They appear to be isolated, and conjecturally, they have the smallest area among $Q$-covers. See \Cref{fig:hexa3} for one such example: $K_0$, $K_1$, $-K_0$, and $-K_1$ are conjecturally the only area-minimizing $Q$-covers there. Observe that the orange triangle slides freely inside $K_1$.

\begin{figure}[ht]
    \centering
    \includegraphics[width=0.8\linewidth]{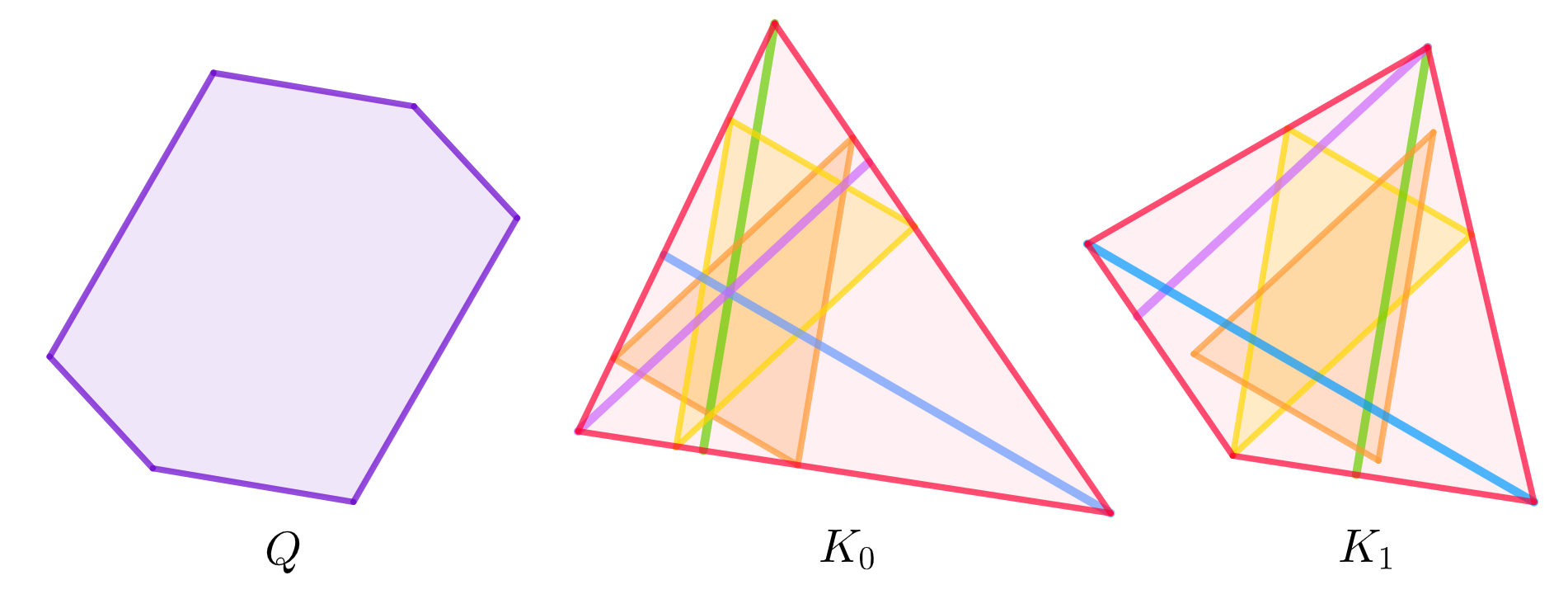}
    \caption{Expected smallest $Q$-covers for a non-regular centrally symmetric hexagon $Q$.}
    \label{fig:hexa3}
\end{figure}

\bibliography{viterbo}
\bibliographystyle{abbrv}
\end{document}